\documentclass[12pt]{article}%
\usepackage{amsmath}
\usepackage{amsfonts}
\usepackage{amssymb}
\usepackage{graphicx}%
\providecommand{\U}[1]{\protect\rule{.1in}{.1in}}
\newtheorem{theorem}{Theorem}

\newtheorem{lemma}[theorem]{Lemma}

\newtheorem{proposition}[theorem]{Proposition}
\newtheorem{remark}[theorem]{Remark}

\newenvironment{proof}[1][Proof]{\noindent\textbf{#1.} }{\ \rule{0.5em}{0.5em}}
\begin{document}

\title{Approximate transmission conditions\\for a Poisson problem at mid-diffusion}
\author{Khaled El-Ghaouti BOUTARENE\\{\small USTHB, Faculty of Mathematics, AMNEDP Laboratory,}\\{\small PO Box 32, El Alia 16111, Bab Ezzouar, Algiers , Algeria} \\{\small kboutarene@usthb.dz,} {\small boutarenekhaled@yahoo.fr}}
\maketitle

\begin{abstract}
This work consists in the asymptotic analysis of the solution of Poisson
equation in a bounded domain of $\mathbb{R}^{P}$ $(P=2,3)$ with a thin layer.
We use a method based on hierarchical variational equations to derive
asymptotic expansion of the solution with respect to the thickness of the thin
layer. We determine the first two terms of the expansion and prove the error
estimate made by truncating the expansion after a finite number of terms.
Next, using the first two terms of the asymptotic expansion, we show that we
can model the effect of the thin layer by a problem with transmission
conditions of order two.

\end{abstract}


\section{Introduction}

\label{}

This paper deals with the study of the asymptotic behavior of the solution of
Poisson equation in a bounded domain $\Omega$ of $\mathbb{R}^{P}$ ($P=2,3 $)
consisting of two sub-domains separated by a thin layer of thickness $\delta$
(destined to tend to 0). The mesh of these thin geometries presents numerical
instabilities that can severely damage the accuracy of the entire process of
resolution. To overcome this difficulty, we adopt asymptotic methods to model
the effect of the thin layer by problems with either appropriate boundary
conditions when we consider a domain surrounded by a thin layer (see for
instance
\cite{ammari-nedelec-1996,benlem96,benlem08,Durufle-Haddar-Joly-2006,engned})
or, as in this paper, with suitable transmission conditions on the interface
(see for instance
\cite{bouta1,delou,clair2.2,clair2.1,SchTor2010,SchTor-magne2010}). Although
this type of conditions has been widely studied, there is still a lot to be
understood concerning the effects of thin shell and their modelisation. Our
motivation comes from \cite{clair-phd,SchTor2010}, in which the authors have
worked on problems of electromagnetic and biological origins. We cite for
example that of Poignard \cite[Chapter 2]{clair-phd}. He considered a cell
immersed in an ambient medium and studied the electric field in the transverse
magnetic (TM) mode at mid-frequency and from which our problem was inspired.

Let us give now precise notations. Let $\Omega$ be a bounded domain of
$\mathbb{R}^{P}$ ($P=2,3$) consisting of three smooth sub-domains: an open
bounded subset $\Omega_{i,\delta}$ with regular boundary $\Gamma_{\delta,1}$,
an exterior domain $\Omega_{e,\delta}$ with disjoint regular boundaries
$\Gamma_{\delta,2}$ and $\partial\Omega$, and a membrane $\Omega_{\delta}$
(thin layer) of thickness $\delta$ separating $\Omega_{i,\delta}$ from
$\Omega_{e,\delta}\ $(see Fig. \ref{fig1}).
\begin{figure}[th]
\centering {\ \begin{minipage}[t]{6cm}
{\includegraphics[width=6cm]{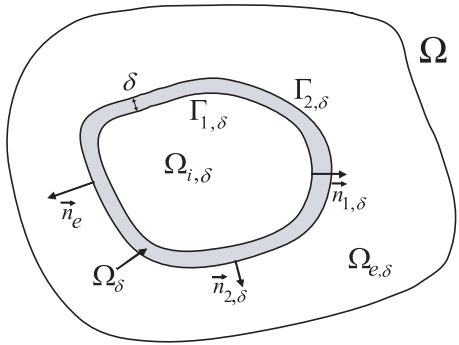}}
\caption{Geometric data}\label{fig1}
\end{minipage}
}\end{figure}

\bigskip

Define the piecewise regular function $\alpha$ by
\[
\alpha(x)=\left\{
\begin{array}
[c]{ll}%
\alpha_{e} & \text{if }x\in\Omega_{e,\delta},\\
\alpha_{\delta} & \text{if }x\in\Omega_{\delta},\\
\alpha_{i} & \text{if }x\in\Omega_{i,\delta},
\end{array}
\right.
\]
where $\alpha_{e},$ $\alpha_{\delta}$ and $\alpha_{i}$ are strictly positive
constants satisfying $\alpha_{i}<\alpha_{\delta}<\alpha_{e}$ or $\alpha
_{e}<\alpha_{\delta}<\alpha_{i}$ which correspond to the case of
mid-diffusion. For a given $f$ in $\mathcal{C}^{\infty}(\Omega),$ we are
interested in the unique solution $u_{\delta}:=(u_{i,\delta},u_{d,\delta
},u_{e,\delta})$ in $H^{1}(\Omega)$ of the following diffusion problem
\begin{subequations}
\label{1}%
\begin{equation}
\left\{
\begin{array}
[c]{ll}%
-div\left(  \alpha\nabla u_{\delta}\right)  =f & \text{in }\Omega,\\
u_{\delta|\partial\Omega}=0 & \text{on }\partial\Omega,
\end{array}
\right. \label{1.01}%
\end{equation}
with transmission conditions on the interfaces%
\begin{equation}
\left\{
\begin{array}
[c]{ll}%
u_{d,\delta|\Gamma_{\delta,2}}=u_{e,\delta|\Gamma_{\delta,2}} & \text{on
}\Gamma_{\delta,2},\\
\alpha_{\delta}\partial_{\mathbf{n}_{\delta,2}}u_{d,\delta|\Gamma_{\delta,2}%
}=\alpha_{e}\partial_{\mathbf{n}_{\delta,2}}u_{e,\delta|\Gamma_{\delta,2}} &
\text{on }\Gamma_{\delta,2},\\
u_{i,\delta|\Gamma_{\delta,1}}=u_{d,\delta|\Gamma_{\delta,1}} & \text{on
}\Gamma_{\delta,1},\\
\alpha_{i}\partial_{\mathbf{n}_{\delta,1}}u_{i,\delta|\Gamma_{\delta,1}%
}=\alpha_{\delta}\partial_{\mathbf{n}_{\delta,1}}u_{d,\delta|\Gamma_{\delta
,1}} & \text{on }\Gamma_{\delta,1},
\end{array}
\right. \label{1.02}%
\end{equation}
where $\partial_{\mathbf{n}_{e}},\partial_{\mathbf{n}_{\delta,2}}$ and
$\partial_{\mathbf{n}_{\delta,1}}$ denote the derivatives in the direction of
the unit normal vectors $\mathbf{n}_{e},\mathbf{n}_{\delta,2}$ and
$\mathbf{n}_{\delta,1}$ to $\partial\Omega,\Gamma_{\delta,2}$ and
$\Gamma_{\delta,1}$ respectively (see Fig. \ref{fig1}).

The main result of this paper is to approximate the solution $u_{\delta}$ of
Problem (\ref{1}) by a solution of a problem involving Poisson equation in
$\Omega$ with two sub-domains separated by an arbitrary interface $\Gamma$
between $\Gamma_{\delta,1}$ and $\Gamma_{\delta,2}$ (see Fig. \ref{fig2} and
Fig. \ref{fig3}), with transmission conditions of order two on $\Gamma$,
modeling the effect of the thin layer. However, it seems that the existence
and uniqueness of the solution of this problem are not obvious$.$ Therefore,
we rewrite the problem into a pseudodifferential equation (cf. \cite{bonnvial}%
) and show that in the case of mid-diffusion, we can find the appropriate
position of the surface $\Gamma$ to solve this equation. The cases 3D and 2D
are similar. We treat the three-dimensional case and the two dimensional one
comes as a remark.

The present paper is organized as follows. In Section 2, we give the statement
of the model problem considered. In section 3, we collect basic results of
differential geometry of surfaces. Sections 4 and 5 are devoted to the
asymptotic analysis of our problem. We present, in section 4, hierarchical
variational equations suited to the construction of a formal asymptotic
expansion up to any order, while Section 5 focuses on the convergence of this
ansatz. With the help of the asymptotic expansion of the solution $u_{\delta}%
$, we model, in the last section, the effect of the thin layer by a problem
with appropriate transmission conditions.

\section{Problem setting}

\begin{figure}[ptbh]
\begin{minipage}[c]{.46\linewidth}
\begin{center}
\includegraphics[width=6cm]{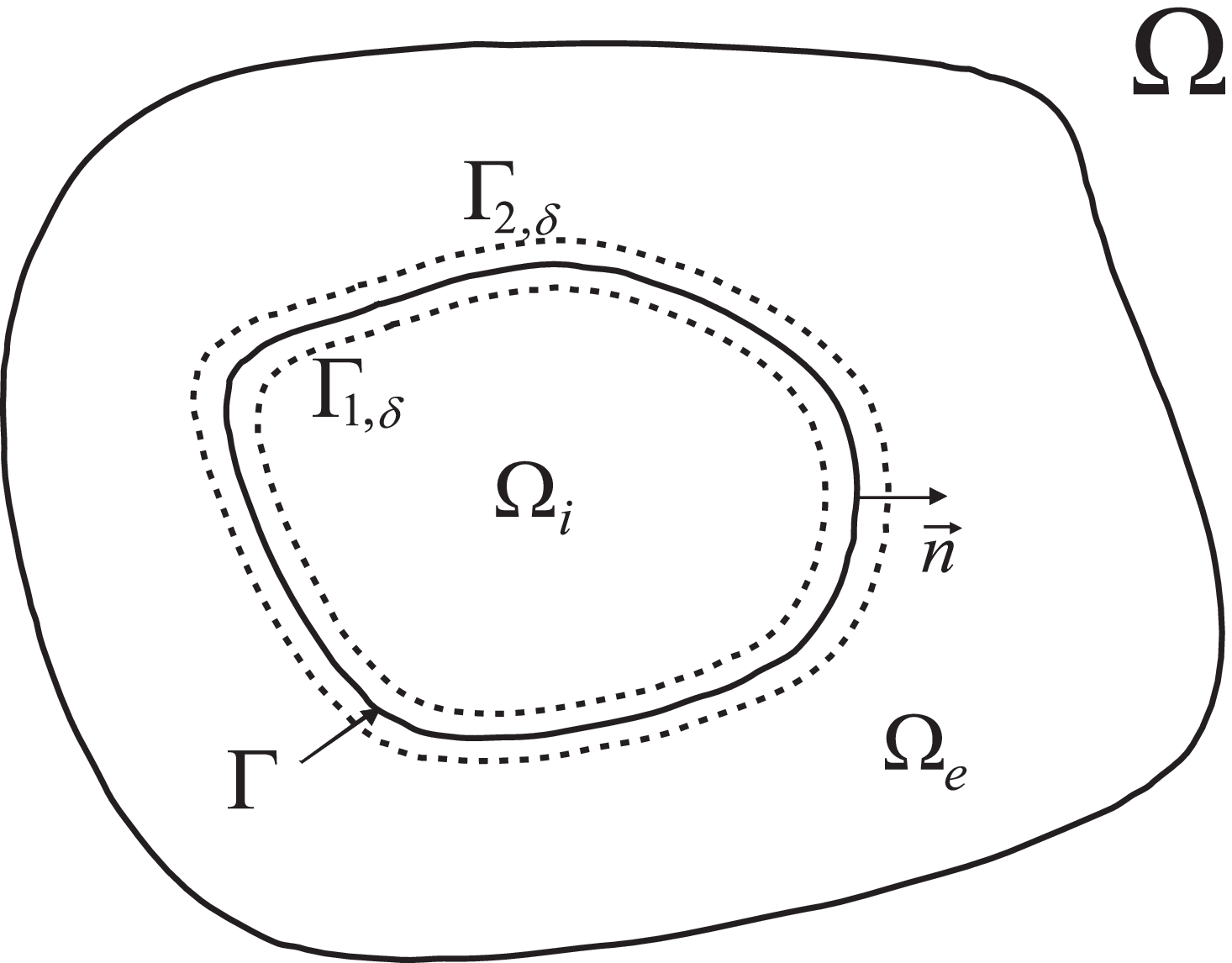}
\caption{Geometry of the studied problem}
\label{fig2}
\end{center}
\end{minipage}
\hfill\begin{minipage}[c]{.46\linewidth}
\begin{center}
\includegraphics[width=3.5cm]{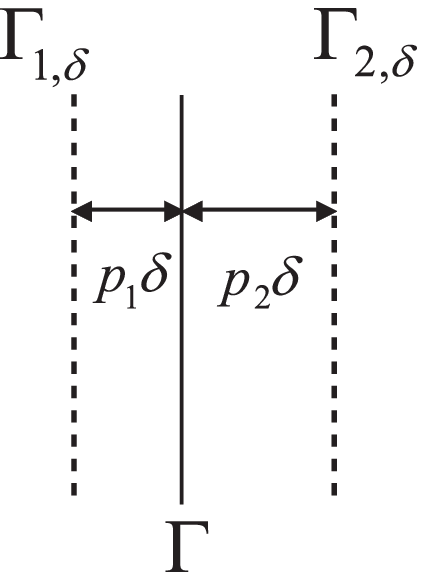}
\caption{The thin layer $\Omega _{\delta }$}
\label{fig3}
\end{center}
\end{minipage}
\end{figure}
We consider a parallel surface $\Gamma$ to $\Gamma_{\delta,1}$ and
$\Gamma_{\delta,2}$ dividing $\Omega_{\delta}$ into two thin layers
$\Omega_{\delta,1}$ and $\Omega_{\delta,2}$ of thickness respectively
$p_{1}\delta$ and $p_{2}\delta,$ where $p_{1}$ and $p_{2}$ are nonnegative
real numbers satisfying $p_{1}+p_{2}=1$ and such that $p_{1}$ and $p_{2}$
belong to a small neighborhood of $1/2$ (see Fig. \ref{fig2} and Fig.
\ref{fig3}). The term \textit{small} neighborhood means that the constants
$p_{1} $ and $p_{2}$ are not too close to $1$ or $0$, in order to avoid having
a layer too thin compared to the other because the following analysis does not
lend itself to this case. Under the aforementioned assumptions, we investigate
in $H^{1}(\Omega)$ the solution $u_{\delta}:=(u_{i,\delta},u_{d_{1},\delta
},u_{d_{2},\delta},u_{e,\delta})$ of the following problem
\end{subequations}
\begin{subequations}
\label{1.1}%
\begin{align}
-div\left(  \alpha\nabla u_{\delta}\right)   &  =f\hspace{3mm}\text{ in
}\Omega,\label{1.1.1}\\
u_{\delta|\partial\Omega} &  =0\hspace{3mm}\text{ on }\partial\Omega
,\label{1.1.2}%
\end{align}
with transmission conditions%
\begin{eqnarray}
u_{d_{2},\delta|\Gamma_{\delta,2}} &  =u_{e,\delta|\Gamma
_{\delta,2}}\ \ \ \ \ \ \ \ \ \ \ \ \ \ \ \text{on }\Gamma_{\delta
,2},\label{1.1.6}\\
\alpha_{\delta}\partial_{\mathbf{n}_{\delta,2}}u_{d_{2},\delta|\Gamma
_{\delta,2}} &  =\alpha_{e}\partial_{\mathbf{n}_{\delta,2}}u_{e,\delta
|\Gamma_{\delta,2}}\hspace{3mm}\ \ \ \ \text{on }\Gamma_{\delta,2}%
,\label{1.1.7}\\
u_{d_{1},\delta|\Gamma} &  =u_{d_{2},\delta|\Gamma}\hspace{3mm}%
\ \ \ \ \ \ \ \ \ \ \ \ \ \ \text{on }\Gamma,\label{1.1.8}\\
\alpha_{\delta}\partial_{\mathbf{n}}u_{d_{1},\delta|\Gamma} &  =\alpha
_{\delta}\partial_{\mathbf{n}}u_{d_{2},\delta|\Gamma}\ \ \ \ \ \ \ \ \hspace
{3mm}\text{on }\Gamma,\label{1.1.9}\\
u_{i,\delta|\Gamma_{\delta,1}} &  =u_{d_{1},\delta|\Gamma_{\delta,1}}%
\hspace{0.06in}\ \ \ \ \ \ \ \ \ \ \ \ \ \text{on }\Gamma_{\delta
,1},\label{1.1.10}\\
\alpha_{i}\partial_{\mathbf{n}_{\delta,1}}u_{i,\delta|\Gamma_{\delta,1}} &
=\alpha_{\delta}\partial_{\mathbf{n}_{\delta,1}}u_{d_{1},\delta|\Gamma
_{\delta,1}}\ \ \ \ \ \text{on }\Gamma_{\delta,1},\label{1.1.11}%
\end{eqnarray}
where $\partial_{\mathbf{n}}$ denotes the derivative in the direction of the
unit normal vector $\mathbf{n}$ to $\Gamma$ (outer for $\Omega_{\delta,1}$ and
inner for $\Omega_{\delta,2}$).

\section{Notations and definitions}

The goal of this section is to define and to collect the main features of
differential geometry \cite{docarmo} (see also \cite{nedelec}) in order to
formulate our problem in a fixed domain (independent of $\delta$) which is a
key tool to determine the asymptotic expansion of the solution $u_{\delta}$.

In the sequel, Greek indice $\beta$ takes the values 1 and 2. Let
$I_{\delta,1}=\left(  -\delta,0\right)  $ and $I_{\delta,2}=\left(
0,\delta\right)  .$ We parameterize the thin shell $\Omega_{\delta,\beta}$ by
the manifold $\Gamma\times I_{\delta,\beta}$ through the mapping $\psi_{\beta
}$ defined by
\end{subequations}
\[
\left\{
\begin{array}
[c]{rcl}%
\Gamma\times I_{\delta,\beta} & \overset{\psi_{\beta}}{\rightarrow} &
\Omega_{\delta,\beta}\\
(m,\eta_{\beta}) & \rightarrow & x:=m+p_{\beta}\eta_{\beta}\mathbf{n}(m).
\end{array}
\right.
\]
As well-known \cite{docarmo}, if the thickness of $\Omega_{\delta,\beta}$ is
small enough, $\psi_{\beta}$ is a $C^{\infty}$-diffeomorphism of manifolds and
it is also known \cite[Remark 2.1]{clair2.2} that the normal vector
$\mathbf{n}_{\delta,\beta}$ to $\Gamma_{\delta,\beta}$ can be identified to
$\mathbf{n}$. To each function $v_{\beta}$ defined on $\Omega_{\delta,\beta}$,
we associate the function $\widetilde{v}_{\beta}$ defined on $\Gamma\times
I_{\delta,\beta}$ by%
\[
\left\{
\begin{array}
[c]{rl}%
\widetilde{v}_{\beta}(m,\eta_{\beta}) & :=v_{\beta}(x),\\
x & =\psi_{\beta}\left(  m,\eta_{\beta}\right),
\end{array}
\right.
\]
then, we have
\[
\nabla v_{\beta}=\left(  I+p_{\beta}\eta_{\beta}\mathcal{R}\right)
^{-1}\nabla_{\Gamma}\widetilde{v}_{\beta}+p_{\beta}^{-1}\frac{\partial
\widetilde{v}_{\beta}}{\partial\eta_{\beta}}\mathbf{n},
\]
where $\nabla_{\Gamma}\widetilde{v}_{\beta}(m)$ and $\mathcal{R}$ are
respectively the surfacic gradient of $\widetilde{v}$ at $m\in\Gamma$ and the
curvature operator $\mathcal{R}$ of $\Gamma$ at point $m.$ The volume element
on the thin shell $\Omega_{\delta,\beta}$ is given by%
\[
d\Omega_{\delta,\beta}=p_{\beta}\det\left(  I+p_{\beta}\eta_{\beta}%
\mathcal{R}\right)  \ d\Gamma d\eta_{\beta}.
\]
Now, we introduce the scaling $s_{\beta}=\eta_{\beta}/\delta,$ and the
intervals $I_{1}=\left(  -1,0\right)  $ and $I_{2}=\left(  0,1\right)  $ such
that the $C^{\infty}$-diffeomorphism $\Phi_{\beta},$ defined by
\[
\left\{
\begin{array}
[c]{rcl}%
\Omega^{\beta}:=\Gamma\times I_{\beta} & \overset{\Phi_{\beta}}{\rightarrow} &
\Omega_{\delta,\beta}\\
(m,s_{\beta}) & \rightarrow & x:=m+\delta p_{\beta}s_{\beta}\mathbf{n}(m),
\end{array}
\right.
\]
parameterizes the thin shell $\Omega_{\delta,\beta}.$ To any function
$v_{\beta}$ defined on $\Omega_{\delta,\beta}$, we associate the function
$v^{[\beta]}$ defined on $\Omega^{\beta}$ through%
\[
\left\{
\begin{array}
[c]{rl}%
v^{[\beta]}(m,s_{\beta}) & :=v_{\beta}(x),\\
x & =\Phi_{\beta}(m,s_{\beta}),
\end{array}
\right.
\]
then the gradient takes the form%
\begin{equation}
\nabla v_{\beta}=\left(  I+\delta p_{\beta}s_{\beta}\mathcal{R}\right)
^{-1}\nabla_{\Gamma}v^{[\beta]}+p_{\beta}^{-1}\delta^{-1}\frac{\partial
v^{[\beta]}}{\partial s_{\beta}}\mathbf{n}.\label{2}%
\end{equation}
The volume element on the thin shell $\Omega_{\delta,\beta}$ becomes
\begin{equation}
d\Omega_{\delta,\beta}=p_{\beta}\delta\det J_{\delta,\beta}\ d\Gamma
ds_{\beta},\label{3}%
\end{equation}
where%
\[
J_{\delta,\beta}:=I+p_{\beta}\delta s_{\beta}\mathcal{R}.
\]
Let $u_{\beta}$ and $v_{\beta}$ be two regular functions defined on
$\Omega_{\delta,\beta}$. From (\ref{2}) and (\ref{3}), we get the change of
variables formula
\begin{align}
\int_{\Omega_{\delta,\beta}}\nabla u_{\beta}.\nabla v_{\beta}\ d\Omega
_{\delta,\beta}  &  =p_{\beta}\delta\int_{\Omega^{\beta}}J_{\delta,\beta}%
^{-2}\nabla_{\Gamma}u^{[\beta]}.\nabla_{\Gamma}v^{[\beta]}\det J_{\delta
,\beta}\ d\Gamma ds_{\beta}\nonumber\\
&  +p_{\beta}^{-1}\delta^{-1}\int_{\Omega^{\beta}}\partial_{s_{\beta}%
}u^{[\beta]}\partial_{s_{\beta}}v^{[\beta]}\det J_{\delta,\beta}\ d\Gamma
ds_{\beta}.\label{4}%
\end{align}

\begin{remark}
In the two-dimensional case, if $m\in\Gamma$, we parameterize the curve
$\Gamma$ by $m(t)$ where $t\in(0,l_{\Gamma})$ is the curvilinear abscissa and
$l_{\Gamma}$ is the length of the curve $\Gamma$, then formula (\ref{4}) turns
into%
\begin{align}
\int_{\Omega_{\delta,\beta}}\nabla u_{\beta}.\nabla v_{\beta}\ d\Omega
_{\delta,\beta}  &  =p_{\beta}^{-1}\delta^{-1}\int_{\Omega^{\beta}}\left(
1+p_{\beta}\delta s_{\beta}\mathcal{R}\right)  \partial_{s_{\beta}}u^{[\beta
]}\partial_{s_{\beta}}v^{[\beta]}\ d\Gamma ds_{\beta}\nonumber\\
&  +p_{\beta}\delta\int_{\Omega^{\beta}}\left(  1+p_{\beta}\delta s_{\beta
}\mathcal{R}\right)  ^{-1}\partial_{t}u^{[\beta]}\partial_{t}v^{[\beta
]}\ d\Gamma ds_{\beta}.\label{6}%
\end{align}

\end{remark}

\begin{remark}
For any function $u$ defined in a neighborhood of $\Gamma,$ we denote, for
convenience, by $u_{|\Gamma}$ the trace of $u$ on $\Gamma$ indifferently in
local coordinates or in Cartesian coordinates.
\end{remark}

\section{The asymptotic analysis}

This section is devoted to the asymptotic analysis of the solution of Problem
(\ref{1.1}). We show that this latter is equivalent to a variational equation
from which we derive the asymptotic expansion of $u_{\delta}.$ We give a
hierarchy of variational equations needed to determine the terms of the
expansion and we calculate the first two terms of the expansion.

Let $v_{d}$ be in $H^{1}(\Omega_{\delta}).$ We denote by $v_{d_{\beta}}$ its
restriction to $\Omega_{\delta,\beta}$. Multiplying Equation
\begin{subequations}
\[
-div\left(  \alpha_{\delta}\nabla u_{d,\delta}\right)  =f_{|\Omega_{\delta}%
}\hspace{3mm}\text{ in }\Omega_{\delta},
\]
by test functions $v_{d}$, using (\ref{1.1.7}), (\ref{1.1.9}), (\ref{1.1.11})
and Green's formula, we get
\end{subequations}
\begin{gather*}
\left\langle \alpha_{i}\partial_{\mathbf{n}_{\delta,1}}u_{i,\delta
|\Gamma_{\delta,1}},v_{d_{1}}\right\rangle _{H^{-1/2}(\Gamma_{\delta,1})\times
H^{1/2}(\Gamma_{\delta,1})}+\alpha_{\delta}\int_{\Omega_{\delta}}\nabla
u_{d,\delta}.\nabla v_{d}\ d\Omega_{\delta}\\
-\left\langle \alpha_{e}\partial_{\mathbf{n}_{\delta,2}}u_{e,\delta
|\Gamma_{\delta,2}},v_{d_{2}}\right\rangle _{H^{-1/2}(\Gamma_{\delta,2})\times
H^{1/2}(\Gamma_{\delta,2})}=\int_{\Omega_{\delta}}f_{|\Omega_{\delta}}%
v_{d}\ d\Omega_{\delta},
\end{gather*}
in which $\left\langle .,.\right\rangle _{H^{-1/2}(\Gamma_{\delta,\beta
})\times H^{1/2}(\Gamma_{\delta,\beta})}$ denotes the duality pairing between
$H^{-1/2}(\Gamma_{\delta,\beta})$ and $H^{1/2}(\Gamma_{\delta,\beta})$. We use
the dilation in the thin layer and Formula (\ref{4}), to obtain%
\begin{gather}
\left\langle \alpha_{i}\partial_{\mathbf{n}_{\delta,1}}u_{i,\delta
|\Gamma_{\delta,1}}\circ\Phi_{1}(m,-1),v^{\left[  1\right]  }%
(m,-1)\right\rangle _{H^{-1/2}(\Gamma\times\left\{  -1\right\}  )\times
H^{1/2}(\Gamma\times\left\{  -1\right\}  )}\nonumber\\
-\left\langle \alpha_{e}\partial_{\mathbf{n}_{\delta,2}}u_{e,\delta
|\Gamma_{\delta,2}}\circ\Phi_{2}(m,1),v^{\left[  2\right]  }(m,1)\right\rangle
_{H^{-1/2}(\Gamma\times\left\{  1\right\}  )\times H^{1/2}(\Gamma
\times\left\{  1\right\}  )}\nonumber\\
+\displaystyle\sum_{\beta=1}^{2}\left[  \alpha_{\delta}\delta a_{\delta
}^{\left[  \beta\right]  }\left(  u_{d,\delta}^{\left[  \beta\right]
},v^{\left[  \beta\right]  }\right)  \right]  =\int_{\Omega_{\delta}%
}f_{|\Omega_{\delta}}v_{d}\ d\Omega_{\delta},\label{12}%
\end{gather}
which is the starting point for the asymptotic analysis, where the bilinear
form $a^{\left[  \beta\right]  }\left(  .,.\right)  $ is defined by%
\begin{align}
a_{\delta}^{\left[  \beta\right]  }\left(  u^{[\beta]},v^{[\beta]}\right)   &
:=p_{\beta}\int_{\Omega^{\beta}}J_{\delta,\beta}^{-2}\nabla_{\Gamma}%
u^{[\beta]}.\nabla_{\Gamma}v^{[\beta]}\det J_{\delta,\beta}\ d\Gamma
ds_{\beta}\nonumber\\
&  +p_{\beta}^{-1}\delta^{-2}\int_{\Omega^{\beta}}\partial_{s_{\beta}%
}u^{[\beta]}\partial_{s_{\beta}}v^{[\beta]}\det J_{\delta,\beta}\ d\Gamma
ds_{\beta},\label{10}%
\end{align}
for every $u^{[\beta]}$ and $v^{[\beta]}$ in $H^{1}(\Omega^{\beta}).$

\subsection{Hierarchy of the variational equations}

In the spirit of \cite{clair2.2,SchTor2010}, we will consider two asymptotic
expansions. Exterior expansions corresponding to the asymptotic expansion of
$u_{\delta}$ restricted to $\Omega_{e,\delta}$ and to $\Omega_{i,\delta} $ and
characterized by the ansatz%
\begin{align}
u_{e,\delta} &  =u_{e,0}+\delta u_{e,1}+\cdots,\label{13}\\
u_{i,\delta} &  =u_{i,0}+\delta u_{i,1}+\cdots,\label{14}%
\end{align}
\ where the terms $u_{e,n}$ and $u_{i,n}$ $(n\in\mathbb{N})$ are independent
of $\delta$ and defined on $\Omega_{e}:=\Omega_{e,\delta}\cup\Gamma_{\delta
,2}\cup\Omega_{\delta,2},$ and on $\Omega_{i}:=\Omega_{i,\delta}\cup
\Gamma_{\delta,1}\cup\Omega_{\delta,1}$ which are respectively the limits of
$\Omega_{e,\delta}$ and $\Omega_{i,\delta}$ for $\delta\rightarrow0.$ They
fulfill%
\begin{equation}
\left\{
\begin{array}
[c]{ll}%
-div\left(  \alpha_{i}\nabla u_{i,n}\right)  =\delta_{0,n}f_{|\Omega_{i}} &
\text{in }\Omega_{i}\text{,}\\
-div\left(  \alpha_{e}\nabla u_{e,n}\right)  =\delta_{0,n}f_{|\Omega_{e}} &
\text{in }\Omega_{e}\text{,}\\
u_{e,n|\partial\Omega}=0 & \text{on}\ \partial\Omega,
\end{array}
\right.  \label{15}%
\end{equation}
where $\delta_{0,n}$ indicates the Kronecker symbol, and an interior expansion corresponding to the asymptotic expansion of
$u_{d_{\beta},\delta}$ written in a fixed domain and defined by the ansatz
\begin{equation}
u_{d,\delta}^{\left[  \beta\right]  }=u_{0}^{\left[  \beta\right]  }+\delta
u_{1}^{\left[  \beta\right]  }+\cdots,\text{ in }\Omega^{\beta},\label{16}%
\end{equation}
where the terms $u_{n}^{\left[  \beta\right]  },\ n\in\mathbb{N}$, are
independent of $\delta$. Using a Taylor expansion in the normal variable, we
infer formally%
\begin{align}
u_{i,\delta}\circ\Phi_{1}(m,s_{1}) &  =u_{i,0|\Gamma}+\delta(u_{i,1|\Gamma
}+s_{1}p_{1}\partial_{\mathbf{n}}u_{i,0|\Gamma})\nonumber\\
&  +\delta^{2}(u_{i,2|\Gamma}+s_{1}p_{1}\partial_{\mathbf{n}}u_{i,1|\Gamma
}+\frac{s_{1}^{2}}{2}p_{1}^{2}\partial_{\mathbf{n}}^{2}u_{i,0|\Gamma}%
)+\cdots,\nonumber\\
&  :=U_{i,0}+\delta U_{i,1}+\delta^{2}U_{i,2}+\cdots,\label{16.1}\\
u_{e,\delta}\circ\Phi_{2}(m,s_{2}) &  =u_{e,0|\Gamma}+\delta(u_{e,1|\Gamma
}+s_{2}p_{2}\partial_{\mathbf{n}}u_{e,0|\Gamma})\nonumber\\
&  +\delta^{2}(u_{e,2|\Gamma}+s_{2}p_{2}\partial_{\mathbf{n}}u_{e,1|\Gamma
}+\frac{s_{2}^{2}}{2}p_{2}^{2}\partial_{\mathbf{n}}^{2}u_{e,0|\Gamma}%
)+\cdots,\nonumber\\
&  :=U_{e,0}+\delta U_{e,1}+\delta^{2}U_{e,2}+\cdots.\label{16.2}%
\end{align}
and Transmission Conditions (\ref{1.1.6}) and (\ref{1.1.10}) become%
\begin{align}
u_{e,0|\Gamma}+\delta(u_{e,1|\Gamma}+p_{2}\partial_{\mathbf{n}}u_{e,0|\Gamma
})+\cdots & =u_{0|s_{2}=1}^{\left[  2\right]  }+\delta u_{1|s_{2}=1}^{\left[
2\right]  }+\cdots,\label{18}\\
u_{i,0|\Gamma}+\delta(u_{i,1|\Gamma}-p_{1}\partial_{\mathbf{n}}u_{i,0|\Gamma
})+\cdots & =u_{0|s_{1}=-1}^{\left[  1\right]  }+\delta u_{1|s_{1}%
=-1}^{\left[  1\right]  }+\cdots.\label{17}%
\end{align}
As
\[
-div\left(  \alpha_{e}\nabla\left(  \sum\limits_{n\geq0}\delta^{n}%
u_{e,n}\right)  \right)  =f_{|\Omega_{\delta,2}},
\]
thanks to Green's formula, we get
\begin{gather*}
\left\langle \alpha_{e}\partial_{\mathbf{n}}\left(  \sum_{n\geq0}\delta
^{n}u_{e,n|\Gamma}\right)  ,v_{d_{2}|\Gamma}\right\rangle _{H^{-1/2}%
(\Gamma)\times H^{1/2}(\Gamma)}\\
-\left\langle \alpha_{e}\partial_{\mathbf{n}_{\delta,2}}\left(  \sum_{n\geq
0}\delta^{n}u_{e,n|\Gamma_{\delta,2}}\right)  ,v_{d_{2}|\Gamma_{\delta,2}%
}\right\rangle _{H^{-1/2}(\Gamma_{\delta,1})\times H^{1/2}(\Gamma_{\delta,1}%
)}\\
+\alpha_{e}\int_{\Omega_{\delta,2}}\nabla\left(  \sum_{n\geq0}\delta
^{n}u_{e,n}\right)  .\nabla v_{d_{2}}\ d\Omega_{\delta,2}=\int_{\Omega
_{\delta,2}}f_{|\Omega_{\delta,2}}v_{d_{2}}\ d\Omega_{\delta,2}.
\end{gather*}
Using the the scaling $s_{2}=\eta_{2}/\delta,$ we obtain
\begin{gather}
\int_{\Gamma}\alpha_{e}\partial_{\mathbf{n}}\left(  \sum_{n\geq0}\delta
^{n}u_{e,n|\Gamma}\right)  v^{\left[  2\right]  }(m,0)\ d\Gamma+\alpha
_{e}\delta a_{\delta}^{\left[  2\right]  }\left(  \sum_{n\geq0}\delta
^{n}U_{e,n},v^{\left[  2\right]  }\right) \nonumber\\
-\left\langle \alpha_{e}\partial_{\mathbf{n}_{\delta,2}}\left(  \sum_{n\geq
0}\delta^{n}u_{e,n|\Gamma_{\delta,2}}\right)  \circ\Phi_{2}(m,1),v^{\left[
2\right]  }(m,1)\right\rangle _{H^{-1/2}(\Gamma\times\left\{  1\right\}
)\times H^{1/2}(\Gamma\times\left\{  1\right\}  )}\nonumber\\
=\int_{\Omega_{\delta,2}}f_{|\Omega_{\delta,2}}v_{d_{2}}\ d\Omega_{\delta
,2}.\label{18.1}%
\end{gather}
In the same way, we obtain the equation for $\alpha_{i}\partial_{\mathbf{n}%
_{\delta,1}}\left(  \sum_{n\geq0}\delta^{n}u_{i,n|\Gamma_{\delta,1}}\right)
\circ\Phi_{1}(m,-1)$:
\begin{gather}
\alpha_{i}\delta a_{\delta}^{\left[  1\right]  }\left(  \sum_{n\geq0}%
\delta^{n}U_{i,n},v^{\left[  1\right]  }\right)  -\int_{\Gamma}\alpha
_{i}\partial_{\mathbf{n}}\left(  \sum_{n\geq0}\delta^{n}u_{i,n|\Gamma}\right)
v^{\left[  1\right]  }(m,0)\ d\Gamma\nonumber\\
+\left\langle \alpha_{i}\partial_{\mathbf{n}_{\delta,1}}\left(  \sum_{n\geq
0}\delta^{n}u_{i,n|\Gamma_{\delta,1}}\right)  \circ\Phi_{1}(m,-1),v^{\left[
1\right]  }(m,-1)\right\rangle _{H^{-1/2}(\Gamma\times\left\{  -1\right\}
)\times H^{1/2}(\Gamma\times\left\{  -1\right\}  )}\nonumber\\
=\int_{\Omega_{\delta,1}}f_{|\Omega_{\delta,1}}v_{d_{1}}\ d\Omega_{\delta
,1}.\label{18.2}%
\end{gather}
Inserting expansions (\ref{13}), (\ref{14}) and (\ref{16}) in (\ref{12}),
using (\ref{16.1})-(\ref{16.2}) and (\ref{18.1})-(\ref{18.2}), we get%
\begin{gather}
\int_{\Gamma}\alpha_{i}\partial_{\mathbf{n}}\left(  \sum_{n\geq0}\delta
^{n}u_{i,n|\Gamma}\right)  v^{\left[  1\right]  }(m,0)\ d\Gamma-\alpha
_{i}\delta a_{\delta}^{\left[  1\right]  }\left(  \sum_{n\geq0}\delta
^{n}U_{i,n},v^{\left[  1\right]  }\right) \nonumber\\
+\displaystyle\sum_{\beta=1}^{2}\left[  \alpha_{\delta}\delta a_{\delta
}^{\left[  \beta\right]  }\left(  \sum_{n\geq0}\delta^{n}u_{n}^{\left[
\beta\right]  },v^{\left[  \beta\right]  }\right)  \right]  -\alpha_{e}\delta
a_{\delta}^{\left[  2\right]  }\left(  \sum_{n\geq0}\delta^{n}U_{e,n}%
,v^{\left[  2\right]  }\right) \nonumber\\
-\int_{\Gamma}\alpha_{e}\partial_{\mathbf{n}}\left(  \sum_{n\geq0}\delta
^{n}u_{e,n|\Gamma}\right)  v^{\left[  2\right]  }(m,0)\ d\Gamma=0.\label{18.3}%
\end{gather}
Now, we use the identity (see \cite[p. 1680]{benlem96})%
\begin{align*}
J_{\delta,\beta}^{-2} &  :=I-2s_{\beta}p_{\beta}\delta\mathcal{R}+3\left(
p_{\beta}s_{\beta}\delta\mathcal{R}\right)  ^{2}+\cdots+n\left(  -p_{\beta
}s_{\beta}\delta\mathcal{R}\right)  ^{n-1}\\
&  +\left(  -s_{\beta}p_{\beta}\delta\mathcal{R}\right)  ^{n}\left[
nJ_{\delta,\beta}^{-1}+J_{\delta,\beta}^{-2}\right]  .
\end{align*}
Since%
\[
\det J_{\delta,\beta}=1+2p_{\beta}s_{\beta}\delta\mathcal{H}+\left(  p_{\beta
}s_{\beta}\delta\right)  ^{2}\mathcal{K},
\]
where $2\mathcal{H}:=tr\mathcal{R}$ and $\mathcal{K}:=\det\mathcal{R}$ are
respectively the mean and the Gaussian curvatures of the surface $\Gamma$, the
bilinear form $a_{\delta}^{\left[  \beta\right]  }(.,.)$ admits the expansion
\begin{align}
a_{\delta}^{\left[  \beta\right]  }\left(  .,.\right)   &  =\delta^{-2}%
a_{0,2}^{\left[  \beta\right]  }+\delta^{-1}a_{1,2}^{\left[  \beta\right]
}+\left(  a_{2,2}^{\left[  \beta\right]  }+a_{0,1}^{\left[  \beta\right]
}\right)  +\delta a_{1,1}^{\left[  \beta\right]  }+\cdots\nonumber\\
&  +\delta^{n-1}a_{n-1,1}^{\left[  \beta\right]  }+\delta^{n}r_{n}^{\left[
\beta\right]  }\left(  \delta;.,.\right)  ,\label{21}%
\end{align}
where the forms $a_{k,l}^{\left[  \beta\right]  }$ are independent of $\delta$
and are given by
\begin{align*}
a_{0,2}^{\left[  \beta\right]  }\left(  u^{\left[  \beta\right]  },v^{\left[
\beta\right]  }\right)   &  :=\int_{\Omega^{\beta}}p_{\beta}^{-1}%
\partial_{s_{\beta}}u^{\left[  \beta\right]  }\partial_{s_{\beta}}v^{\left[
\beta\right]  }\ d\Gamma ds_{\beta},\\
a_{1,2}^{\left[  \beta\right]  }\left(  u^{\left[  \beta\right]  },v^{\left[
\beta\right]  }\right)   &  :=\int_{\Omega^{\beta}}2\mathcal{H}s_{\beta
}\partial_{s_{\beta}}u^{\left[  \beta\right]  }\partial_{s_{\beta}}v^{\left[
\beta\right]  }\ d\Gamma ds_{\beta},\\
a_{2,2}^{\left[  \beta\right]  }\left(  u^{\left[  \beta\right]  },v^{\left[
\beta\right]  }\right)   &  :=\int_{\Omega^{\beta}}p_{\beta}\mathcal{K}%
s_{\beta}^{2}\partial_{s_{\beta}}u^{\left[  \beta\right]  }\partial_{s_{\beta
}}v^{\left[  \beta\right]  }\ d\Gamma ds_{\beta},\\
a_{0,1}^{\left[  \beta\right]  }\left(  u^{\left[  \beta\right]  },v^{\left[
\beta\right]  }\right)   &  :=\int_{\Omega^{\beta}}p_{\beta}\nabla_{\Gamma
}u^{\left[  \beta\right]  }.\nabla_{\Gamma}v^{\left[  \beta\right]  }\ d\Gamma
ds_{\beta},\\
a_{1,1}^{\left[  \beta\right]  }\left(  u^{\left[  \beta\right]  },v^{\left[
\beta\right]  }\right)   &  :=\int_{\Omega^{\beta}}2p_{\beta}^{2}s_{\beta
}\left(  \mathcal{H}I-\mathcal{R}\right)  \nabla_{\Gamma}u^{\left[
\beta\right]  }.\nabla_{\Gamma}v^{\left[  \beta\right]  }\ d\Gamma ds_{\beta
},\\
a_{2,1}^{\left[  \beta\right]  }\left(  u^{\left[  \beta\right]  },v^{\left[
\beta\right]  }\right)   &  :=\int_{\Omega^{\beta}}p_{\beta}^{3}\left(
\mathcal{K}I-4\mathcal{HR}+3\mathcal{R}^{2}\right)  s_{\beta}^{2}%
\nabla_{\Gamma}u^{\left[  \beta\right]  }.\nabla_{\Gamma}v^{\left[
\beta\right]  }\ d\Gamma ds_{\beta},\\
a_{n-1,1}^{\left[  \beta\right]  }\left(  u^{\left[  \beta\right]
},v^{\left[  \beta\right]  }\right)   &  :=\int_{\Omega^{\beta}}p_{\beta}%
^{n}\left[  \left(  n-2\right)  \mathcal{KR}^{n-3}-\left(  n-1\right)
2\mathcal{HR}^{n-2}\right. \\
&  \left.  +n\mathcal{R}^{n-1}\right]  \left(  -s_{\beta}\right)  ^{n-1}%
\nabla_{\Gamma}u^{\left[  \beta\right]  }.\nabla_{\Gamma}v^{\left[
\beta\right]  }\ d\Gamma ds_{\beta}.
\end{align*}
The form $r_{n}^{\left[  \beta\right]  }(\delta;.,.)$ is the remainder of
Expansion (\ref{21}) and is expressed by
\[
r_{n}^{\left[  \beta\right]  }(\delta;u^{\left[  \beta\right]  },v^{\left[
\beta\right]  }):=\int_{\Omega^{\beta}}\left(  B_{n,\delta}+2\mathcal{H}%
B_{n-1,\delta}+\mathcal{K}B_{n-2,\delta}\right)  s_{\beta}^{n}\nabla_{\Gamma
}u^{\left[  \beta\right]  }.\nabla_{\Gamma}v^{\left[  \beta\right]  }d\Gamma
ds_{\beta},
\]
with%
\[
B_{n,\delta}:=\left\{
\begin{array}
[c]{l}%
\left(  -\mathcal{R}\right)  ^{n}\left(  nJ_{\delta,\beta}^{-1}+J_{\delta
,\beta}^{-2}\right)  \text{ if}\ n\geq0,\\
J_{\delta,\beta}^{-2}\text{ otherwise.}%
\end{array}
\right.
\]

\begin{remark}
In the two-dimensional case, with the help of (\ref{6}), Expansion (\ref{21})
turns into
\[
a_{\delta}^{\left[  \beta\right]  }\left(  .,.\right)  =\delta^{-2}%
a_{0,2}^{\left[  \beta\right]  }+\delta^{-1}a_{1,2}^{\left[  \beta\right]
}+a_{0,1}^{\left[  \beta\right]  }+\delta a_{1,1}^{\left[  \beta\right]
}+\cdots+\delta^{n-1}a_{n-1,1}^{\left[  \beta\right]  }+\delta^{n}%
r_{n}^{\left[  \beta\right]  }\left(  \delta;.,.\right)  ,
\]
with%
\begin{gather*}
a_{n,2}^{\left[  \beta\right]  }\left(  u^{\left[  \beta\right]  },v^{\left[
\beta\right]  }\right)  :=\int_{\Omega^{\beta}}p_{\beta}^{n-1}\left(
s_{\beta}\mathcal{R}\right)  ^{n}\partial_{s_{\beta}}u^{\left[  \beta\right]
}\partial_{s_{\beta}}v^{\left[  \beta\right]  }\ d\Gamma ds_{\beta},\\
a_{n,1}^{\left[  \beta\right]  }\left(  u^{\left[  \beta\right]  },v^{\left[
\beta\right]  }\right)  :=\int_{\Omega^{\beta}}p_{\beta}^{n+1}\left(
-s_{\beta}\mathcal{R}\right)  ^{n}\partial_{t}u^{\left[  \beta\right]
}\partial_{t}v^{\left[  \beta\right]  }\ d\Gamma ds_{\beta},\\
r_{n}^{\left[  \beta\right]  }\left(  \delta;u^{\left[  \beta\right]
},v^{\left[  \beta\right]  }\right)  :=\int_{\Omega^{\beta}}J_{\delta,\beta
}^{-1}\left(  -s_{\beta}\mathcal{R}\right)  ^{n}\partial_{t}u^{\left[
\beta\right]  }\partial_{t}v^{\left[  \beta\right]  }\ d\Gamma ds_{\beta}.
\end{gather*}

\end{remark}

Inserting Expansion (\ref{21}) in (\ref{18.3}) and matching the same powers of
$\delta,$ we obtain the following variational equations, which hold for all
$v=\left(  v^{\left[  1\right]  },v^{\left[  2\right]  }\right)  $ in
$H^{1}\left(  \Gamma\times\left(  -1,1\right)  \right)  ,$
\begin{equation}
a_{0,2}^{\left[  1\right]  }\left(  \alpha_{\delta}u_{0}^{\left[  1\right]
}-\alpha_{i}U_{i,0},v^{\left[  1\right]  }\right)  +a_{0,2}^{\left[  2\right]
}\left(  \alpha_{\delta}u_{0}^{\left[  2\right]  }-\alpha_{e}U_{e,0}%
,v^{\left[  2\right]  }\right)  =0,\label{23}%
\end{equation}%
\begin{gather}
a_{1,2}^{\left[  1\right]  }\left(  \alpha_{\delta}u_{0}^{\left[  1\right]
}-\alpha_{i}U_{i,0},v^{\left[  1\right]  }\right)  +a_{0,2}^{\left[  1\right]
}\left(  \alpha_{\delta}u_{1}^{\left[  1\right]  }-\alpha_{i}U_{i,1}%
,v^{\left[  1\right]  }\right) \nonumber\\
+a_{1,2}^{\left[  2\right]  }\left(  \alpha_{\delta}u_{0}^{\left[  2\right]
}-\alpha_{e}U_{e,0},v^{\left[  2\right]  }\right)  +a_{0,2}^{\left[  2\right]
}\left(  \alpha_{\delta}u_{1}^{\left[  2\right]  }-\alpha_{e}U_{e,1}%
,v^{\left[  2\right]  }\right) \nonumber\\
=\alpha_{e}\int_{\Gamma}\partial_{\mathbf{n}}u_{e,0|\Gamma}v^{\left[
2\right]  }\left(  m,0\right)  \ d\Gamma-\alpha_{i}\int_{\Gamma}%
\partial_{\mathbf{n}}u_{i,0|\Gamma}v^{\left[  1\right]  }\left(  m,0\right)
\ d\Gamma,\label{24}%
\end{gather}%
\begin{gather}
a_{0,2}^{\left[  1\right]  }\left(  \alpha_{\delta}u_{2}^{\left[  1\right]
}-\alpha_{i}U_{i,2},v^{\left[  1\right]  }\right)  +a_{1,2}^{\left[  1\right]
}\left(  \alpha_{\delta}u_{1}^{\left[  1\right]  }-\alpha_{i}U_{i,1}%
,v^{\left[  1\right]  }\right) \nonumber\\
+\left(  a_{2,2}^{\left[  1\right]  }+a_{0,1}^{\left[  1\right]  }\right)
\left(  \alpha_{\delta}u_{0}^{\left[  1\right]  }-\alpha_{i}U_{i,0},v^{\left[
1\right]  }\right) \nonumber\\
+a_{0,2}^{\left[  2\right]  }\left(  \alpha_{\delta}u_{2}^{\left[  2\right]
}-\alpha_{e}U_{e,2},v^{\left[  2\right]  }\right)  +a_{1,2}^{\left[  2\right]
}\left(  \alpha_{\delta}u_{1}^{\left[  2\right]  }-\alpha_{e}U_{e,1}%
,v^{\left[  2\right]  }\right) \nonumber\\
+\left(  a_{2,2}^{\left[  2\right]  }+a_{0,1}^{\left[  2\right]  }\right)
\left(  \alpha_{\delta}u_{0}^{\left[  2\right]  }-\alpha_{e}U_{e,0},v^{\left[
2\right]  }\right) \nonumber\\
=\alpha_{e}\int_{\Gamma}\partial_{\mathbf{n}}u_{e,1|\Gamma}v^{\left[
2\right]  }\left(  m,0\right)  \ d\Gamma-\alpha_{i}\int_{\Gamma}%
\partial_{\mathbf{n}}u_{i,1|\Gamma}v^{\left[  1\right]  }\left(  m,0\right)
\ d\Gamma,\label{25}%
\end{gather}%
\begin{gather}
\alpha_{i}\int_{\Gamma}\partial_{\mathbf{n}}u_{i,2|\Gamma}v^{\left[  1\right]
}\left(  m,0\right)  \ d\Gamma+a_{0,2}^{\left[  1\right]  }\left(
\alpha_{\delta}u_{3}^{\left[  1\right]  }-\alpha_{i}U_{i,3},v^{\left[
1\right]  }\right) \nonumber\\
+a_{1,2}^{\left[  1\right]  }\left(  \alpha_{\delta}u_{2}^{\left[  1\right]
}-\alpha_{i}U_{i,2},v^{\left[  1\right]  }\right)  +\left(  a_{2,2}^{\left[
1\right]  }+a_{0,1}^{\left[  1\right]  }\right)  \left(  \alpha_{\delta}%
u_{1}^{\left[  1\right]  }-\alpha_{i}U_{i,1},v^{\left[  1\right]  }\right)
\nonumber\\
+a_{1,1}^{\left[  1\right]  }\left(  \alpha_{\delta}u_{0}^{\left[  1\right]
}-\alpha_{i}U_{i,0},v^{\left[  1\right]  }\right)  +a_{0,2}^{\left[  2\right]
}\left(  \alpha_{\delta}u_{3}^{\left[  2\right]  }-\alpha_{e}U_{e,3}%
,v^{\left[  2\right]  }\right) \nonumber\\
+a_{1,2}^{\left[  2\right]  }\left(  \alpha_{\delta}u_{2}^{\left[  2\right]
}-\alpha_{e}U_{e,2},v^{\left[  2\right]  }\right)  +\left(  a_{2,2}^{\left[
2\right]  }+a_{0,1}^{\left[  2\right]  }\right)  \left(  \alpha_{\delta}%
u_{1}^{\left[  2\right]  }-\alpha_{e}U_{e,1},v^{\left[  2\right]  }\right)
\nonumber\\
+a_{1,1}^{\left[  2\right]  }\left(  \alpha_{\delta}u_{0}^{\left[  2\right]
}-\alpha_{e}U_{e,0},v^{\left[  2\right]  }\right)  -\alpha_{e}\int_{\Gamma
}\partial_{\mathbf{n}}u_{e,2|\Gamma}v^{\left[  2\right]  }\left(  m,0\right)
\ d\Gamma=0,\label{25.5}%
\end{gather}%
\begin{gather}
\alpha_{i}\int_{\Gamma}\partial_{\mathbf{n}}u_{i,n|\Gamma}v^{\left[  1\right]
}\left(  m,0\right)  \ d\Gamma+a_{0,2}^{\left[  1\right]  }\left(
\alpha_{\delta}u_{n+1}^{\left[  1\right]  }-\alpha_{i}U_{i,n+1},v^{\left[
1\right]  }\right) \nonumber\\
+a_{1,2}^{\left[  1\right]  }\left(  \alpha_{\delta}u_{n}^{\left[  1\right]
}-\alpha_{i}U_{i,n},v^{\left[  1\right]  }\right)  +a_{0,1}^{\left[  1\right]
}\left(  \alpha_{\delta}u_{n-1}^{\left[  1\right]  }-\alpha_{i}U_{i,n-1}%
,v^{\left[  1\right]  }\right) \nonumber\\
+a_{1,1}^{\left[  1\right]  }\left(  \alpha_{\delta}u_{n-2}^{\left[  1\right]
}-\alpha_{i}U_{i,n-2},v^{\left[  1\right]  }\right)  +a_{2,1}^{\left[
1\right]  }\left(  \alpha_{\delta}u_{n-3}^{\left[  1\right]  }-\alpha
_{i}U_{i,n-3},v^{\left[  1\right]  }\right) \nonumber\\
+\sum\limits_{l=3}^{n}a_{l-1,2}^{\left[  1\right]  }\left(  \alpha_{\delta
}u_{n-l}^{\left[  1\right]  }-\alpha_{i}U_{i,n-l},v^{\left[  1\right]
}\right)  +a_{0,2}^{\left[  2\right]  }\left(  \alpha_{\delta}u_{n+1}^{\left[
2\right]  }-\alpha_{i}U_{e,n+1},v^{\left[  2\right]  }\right) \nonumber\\
+a_{1,2}^{\left[  2\right]  }\left(  \alpha_{\delta}u_{n}^{\left[  2\right]
}-\alpha_{i}U_{e,n},v^{\left[  2\right]  }\right)  +a_{0,1}^{\left[  2\right]
}\left(  \alpha_{\delta}u_{n-1}^{\left[  2\right]  }-\alpha_{i}U_{e,n-1}%
,v^{\left[  2\right]  }\right) \nonumber\\
+a_{1,1}^{\left[  2\right]  }\left(  \alpha_{\delta}u_{n-2}^{\left[  2\right]
}-\alpha_{i}U_{e,n-2},v^{\left[  2\right]  }\right)  +a_{2,1}^{\left[
2\right]  }\left(  \alpha_{\delta}u_{n-3}^{\left[  2\right]  }-\alpha
_{i}U_{e,n-3},v^{\left[  2\right]  }\right) \nonumber\\
+\sum\limits_{l=3}^{n}a_{l-1,2}^{\left[  2\right]  }\left(  \alpha_{\delta
}u_{n-l}^{\left[  2\right]  }-\alpha_{i}U_{e,n-l},v^{\left[  2\right]
}\right)  -\alpha_{e}\int_{\Gamma}\partial_{\mathbf{n}}u_{e,n|\Gamma
}v^{\left[  2\right]  }\left(  m,0\right)  d\Gamma=0,\ n\geq3.\label{26}%
\end{gather}

\subsection{Calculation of the first terms}

In this paragraph, we first recall some theoretical results needed for our
calculation. After, we calculate explicitly the first two terms of
Expansions (\ref{13})-(\ref{14}) and (\ref{16}) in order to present a
recursive method to define successively the terms of these expansions.

Let $s$ be a nonnegative real number. We define $PH^{s}(\Omega)$ (see
\cite{clair-perrussel}) the space of functions in $\Omega$, with $H^{s}%
$-regularity in $\Omega_{e}$ and $\Omega_{i}$, as follows%
\[
PH^{s}(\Omega):=\left\{  V=(V_{i},V_{e});V_{i}\in H^{s}(\Omega_{i})\text{ and
}V_{e}\in H^{s}(\Omega_{e})\right\}  ,
\]
equipped with the norm%
\[
\left\Vert V\right\Vert _{PH^{s}(\Omega)}:=\left(  \left\Vert V_{i}\right\Vert
_{H^{s}(\Omega_{i})}+\left\Vert V_{e}\right\Vert _{H^{s}(\Omega_{e})}\right)
^{1/2}.
\]
We need the following theorem. Its proof \cite[p. 122]{clair-phd} is an
application of the reflection principle \cite[p. 147]{Li-Vogelius}.

\begin{theorem}
\label{theo1}Let $G$ belongs to $H^{s}(\Gamma),$ $s\geq-1/2$. Then the
following problem
\[
\left\{
\begin{array}
[c]{ll}%
-div\left(  \alpha_{i}\nabla U_{i}\right)  =0 & \text{in }\Omega_{i},\\
-div\left(  \alpha_{e}\nabla U_{e}\right)  =0 & \text{in }\Omega_{e},\\
U_{i|\Gamma}=U_{e|\Gamma} & \text{on }\Gamma,\\
\alpha_{i}\partial_{\mathbf{n}}U_{i|\Gamma}-\alpha_{e}\partial_{\mathbf{n}%
}U_{e|\Gamma}=G & \text{on }\Gamma,\\
U_{e|\partial\Omega}=0 & \text{on }\partial\Omega,
\end{array}
\right.
\]
admits a unique solution $U=(U_{i},U_{e})$ in $PH^{s+3/2}(\Omega).$ Moreover,
if $m$ is a nonnegative integer, and $s>m+\dfrac{P-1}{2}$. Then
\[
U_{i}\in\mathcal{C}^{m}(\overline{\Omega}_{i})\ \ \ \text{and}\ \ \ U_{e}%
\in\mathcal{C}^{m}(\overline{\Omega}_{e}).
\]

\end{theorem}
We also need the following technical lemma. This construction was motivated by
\cite{benlem96} and its proof is a straightforward verification.

\begin{lemma}
\label{lem1}For $\beta=1,2$, let $q^{\left[  \beta\right]  }$ be a given
function in $L^{2}(\Gamma)$ and let $k^{\left[  \beta\right]  }$ be a
vectorial function in $L^{2}(\Omega^{\beta},\mathbb{C}^{3})$ such that the
partial application $s_{\beta}\rightarrow k^{\left[  \beta\right]
}(.,s_{\beta})$ is valued in the space of vectorial fields tangent to $\Gamma$
and also $div_{\Gamma}k^{\left[  \beta\right]  }\in L^{2}(\Omega^{\beta}).$
Then the solution $h^{\left[  \beta\right]  }$ of the variational equation%
\begin{gather*}
\mathcal{L}^{\left[  \beta\right]  }v^{\left[  \beta\right]  }:=\int
_{\Omega^{\beta}}h^{\left[  \beta\right]  }\partial_{s_{\beta}}v^{\left[
\beta\right]  }\ d\Gamma ds_{\beta}+\int_{\Omega^{\beta}}k^{\left[
\beta\right]  }.\nabla_{\Gamma}v^{\left[  \beta\right]  }\ d\Gamma ds_{\beta
}=0;\\
\forall v^{\left[  \beta\right]  }\in H^{1}(\Omega^{\beta}),\ v^{\left[
\beta\right]  }(.,0)=0,
\end{gather*}
is explicitly given by%
\[
h^{\left[  \beta\right]  }\left(  m,s_{\beta}\right)  =\int_{s_{\beta}%
}^{(-1)^{\beta}}div_{\Gamma}k^{\left[  \beta\right]  }\left(  m,\lambda
\right)  \ d\lambda.
\]
Moreover, if $v^{\left[  \beta\right]  }(.,0)\neq0,$ we have%
\begin{align*}
\mathcal{L}^{\left[  \beta\right]  }v^{\left[  \beta\right]  }  &
=(-1)^{\beta+1}\int_{\Gamma}h^{\left[  \beta\right]  }(m,0)v^{\left[
\beta\right]  }(m,0)\ d\Gamma\\
&  =\int_{\Gamma}\left[  (-1)^{\beta+1}\int_{0}^{(-1)^{\beta}}div_{\Gamma
}k^{\left[  \beta\right]  }\left(  m,s_{\beta}\right)  \ ds_{\beta}\right]
v^{\left[  \beta\right]  }\left(  m,0\right)  \ d\Gamma.
\end{align*}

\end{lemma}

\subsubsection{Term of order 0}

Equation (\ref{23}) implies that $\partial_{s_{\beta}}u_{0}^{\left[
\beta\right]  }=0.$ Using (\ref{18}), (\ref{17}) and (\ref{1.1.8}), we obtain%
\begin{equation}
u_{i,0|\Gamma}=u_{0}^{\left[  1\right]  }(m,s_{1})=u_{0}^{\left[  2\right]
}(m,s_{2})=u_{e,0|\Gamma},\ m\in\Gamma.\label{27}%
\end{equation}
The choice of $v$ such that $v^{\left[  2\right]  }=0$ in (\ref{24}) gives
\[
a_{0,2}^{\left[  1\right]  }\left[  \alpha_{\delta}u_{1}^{\left[  1\right]
}-\alpha_{i}U_{i,1},v^{\left[  1\right]  }\right]  =0.
\]
An int\'{e}gration by parts in $s_{1}$ leads to
\begin{equation}
p_{1}^{-1}\alpha_{\delta}\ \partial_{s_{1}}u_{1}^{\left[  1\right]  }%
=\alpha_{i}\partial_{\mathbf{n}}u_{i,0|\Gamma}.\label{28}%
\end{equation}
Similarly, the choice of $v$ such that $v^{\left[  1\right]  }=0$ in
(\ref{24}) gives%
\[
a_{0,2}^{\left[  2\right]  }(\alpha_{\delta}u_{1}^{\left[  2\right]  }%
-\alpha_{e}U_{e,0},v^{\left[  2\right]  })=0.
\]
We obtain
\begin{equation}
p_{2}^{-1}\alpha_{\delta}\ \partial_{s_{2}}u_{1}^{\left[  2\right]  }%
=\alpha_{e}\partial_{\mathbf{n}}u_{e,0|\Gamma}.\label{29}%
\end{equation}
Therefore
\[
\alpha_{i}\int_{\Gamma}\partial_{\mathbf{n}}u_{i,0|\Gamma}v^{\left[  1\right]
}(m,0)d\Gamma=\alpha_{e}\int_{\Gamma}\partial_{\mathbf{n}}u_{e,0|\Gamma
}v^{\left[  2\right]  }(m,0)d\Gamma.
\]
As $v^{\left[  1\right]  }(m,0)=v^{\left[  2\right]  }(m,0)$,
\begin{equation}
\alpha_{i}\partial_{\mathbf{n}}u_{i,0|\Gamma}=\alpha_{e}\partial_{\mathbf{n}%
}u_{e,0|\Gamma}.\label{30}%
\end{equation}
Let us define $\alpha_{0}$ and $u_{n}$ $(\forall n\in%
\mathbb{N}
)$ by
\[
\alpha_{0}(x)=\left\{
\begin{array}
[c]{ll}%
\alpha_{e} & \text{if }x\in\Omega_{e},\\
\alpha_{i} & \text{if }x\in\Omega_{i},
\end{array}
\right.  \text{ and \ }u_{n}=\left\{
\begin{array}
[c]{ll}%
u_{e,n} & \text{in }\Omega_{e},\\
u_{i,n} & \text{in }\Omega_{i}.
\end{array}
\right.
\]
Therefore, with (\ref{15}), (\ref{27}) and (\ref{30}), $u_{0}$ satisfies the
following problem
\[
\left\{
\begin{array}
[c]{ll}%
-div\left(  \alpha_{0}\nabla u_{0}\right)  =f & \text{in }\Omega,\\
u_{0|\partial\Omega}=0 & \text{on }\partial\Omega.
\end{array}
\right.
\]
Elliptic regularity results (see e.g. \cite{adn2}) show that if $f$ belongs to
$\mathcal{C}^{\infty}(\Omega)$, then $(u_{i,0},u_{e,0})$ is a well-defined
element of $\mathcal{C}^{\infty}(\overline{\Omega}_{i})\times\mathcal{C}%
^{\infty}(\overline{\Omega}_{e}).$ As a consequence the first term is determined.

\subsubsection{Term of order 1}

Integrating Relations (\ref{28}) and (\ref{29}) in $s_{\beta},$ yields%
\begin{align*}
u_{1}^{\left[  1\right]  }(m,s_{1})  & =u_{i,1|\Gamma}+p_{1}\left[
(s_{1}+1)\alpha_{i}\alpha_{\delta}^{-1}-1\right]  \partial_{\mathbf{n}%
}u_{i,0|\Gamma},\ \forall(m,s_{1})\in\Omega^{1},\\
u_{1}^{\left[  2\right]  }(m,s_{2})  & =u_{e,1|\Gamma}+p_{2}\left[
(s_{2}-1)\alpha_{e}\alpha_{\delta}^{-1}+1\right]  \partial_{\mathbf{n}%
}u_{e,0|\Gamma},\ \forall(m,s_{2})\in\Omega^{2}.
\end{align*}
By identifying terms of order 1 in (\ref{17}) and (\ref{18}), we obtain the
first transmission condition on $\Gamma$%
\begin{equation}
u_{i,1|\Gamma}-u_{e,1|\Gamma}=p_{1}(1-\alpha_{i}\alpha_{\delta}^{-1}%
)\partial_{\mathbf{n}}u_{i,0|\Gamma}+p_{2}(1-\alpha_{e}\alpha_{\delta}%
^{-1})\partial_{\mathbf{n}}u_{e,0|\Gamma}.\label{30a}%
\end{equation}
The second one follows the same lines as for order $0.$ The choice of $v$ such
that $v^{\left[  2\right]  }=0$ in (\ref{25}) gives%
\[
a_{0,2}^{\left[  1\right]  }\left(  \alpha_{\delta}u_{2}^{\left[  1\right]
}-\alpha_{i}U_{i,2},v^{\left[  1\right]  }\right)  +a_{0,1}^{\left[  1\right]
}\left(  \alpha_{\delta}u_{0}^{\left[  1\right]  }-\alpha_{i}U_{i,0}%
,v^{\left[  1\right]  }\right)  =0.
\]
We apply Lemma \ref{lem1} with%
\begin{align*}
h^{\left[  1\right]  } &  =p_{1}^{-1}\alpha_{\delta}\ \partial_{s_{1}}%
u_{2}^{\left[  1\right]  }-p_{1}^{-1}\alpha_{i}U_{i,2}=p_{1}^{-1}%
\alpha_{\delta}\ \partial_{s_{1}}u_{2}^{\left[  1\right]  }-\alpha_{i}%
\partial_{\mathbf{n}}u_{i,1|\Gamma}-s_{1}p_{1}\partial_{\mathbf{n}}%
^{2}u_{i,0|\Gamma},\\
k^{\left[  1\right]  } &  =p_{1}\nabla_{\Gamma}\left(  p_{1}^{-1}%
u_{0}^{\left[  1\right]  }-\alpha_{i}U_{i,0}\right)  =p_{1}\left(
\alpha_{\delta}-\alpha_{i}\right)  \nabla_{\Gamma}u_{i,0|\Gamma},
\end{align*}
we find%
\[
p_{1}^{-1}\alpha_{\delta}\ \partial_{s_{1}}u_{2}^{\left[  1\right]  }%
(m,s_{1})-\alpha_{i}\partial_{\mathbf{n}}u_{i,1|\Gamma}-s_{1}p_{1}%
\partial_{\mathbf{n}}^{2}u_{i,0|\Gamma}=-\left(  s_{1}+1\right)  p_{1}\left(
\alpha_{\delta}-\alpha_{i}\right)  \Delta_{\Gamma}u_{i,0|\Gamma}.
\]
Morover, for all $v^{\left[  1\right]  }$ we obtain
\[
\mathcal{L}^{\left[  1\right]  }v^{\left[  1\right]  }=-\int_{\Gamma}%
p_{1}\left(  \alpha_{\delta}-\alpha_{1}\right)  \Delta_{\Gamma}u_{i,0|\Gamma
}\ v^{\left[  1\right]  }(m,0)\ d\Gamma.
\]
Again Similarly, the choice of $v$ such that $v^{\left[  1\right]  }=0$ in
(\ref{25}) gives
\[
a_{0,2}^{\left[  2\right]  }\left(  \alpha_{\delta}u_{2}^{\left[  2\right]
}-\alpha_{e}U_{e,2},v^{\left[  2\right]  }\right)  +a_{0,1}^{\left[  2\right]
}\left(  \alpha_{\delta}u_{0}^{\left[  2\right]  }-\alpha_{e}U_{e,0}%
,v^{\left[  2\right]  }\right)  =0.
\]
We apply Lemma \ref{lem1} with
\begin{align*}
h^{\left[  2\right]  } &  =p_{2}^{-1}\alpha_{\delta}\ \partial_{s_{2}}%
u_{2}^{\left[  2\right]  }-p_{2}^{-1}\alpha_{e}U_{e,2}=p_{2}^{-1}%
\alpha_{\delta}\ \partial_{s_{2}}u_{2}^{\left[  2\right]  }-\alpha_{e}%
\partial_{\mathbf{n}}u_{e,1|\Gamma}-s_{2}p_{2}\partial_{\mathbf{n}}%
^{2}u_{e,0|\Gamma},\\
k^{\left[  2\right]  } &  =p_{2}\nabla_{\Gamma}\left(  \alpha_{\delta}%
u_{0}^{\left[  2\right]  }-\alpha_{e}U_{e,0}\right)  =p_{2}\left(
\alpha_{\delta}-\alpha_{e}\right)  \nabla_{\Gamma}u_{e,0|\Gamma},
\end{align*}
we find%
\[
p_{2}^{-1}\alpha_{\delta}\ \partial_{s_{2}}u_{2}^{\left[  2\right]  }%
(m,s_{2})-\alpha_{e}\partial_{\mathbf{n}}u_{e,1|\Gamma}-s_{2}p_{2}%
\partial_{\mathbf{n}}^{2}u_{e,0|\Gamma}=\left(  1-s_{2}\right)  \left(
\alpha_{\delta}-\alpha_{e}\right)  \Delta_{\Gamma}u_{e,0|\Gamma}.
\]
Morover, for all $v^{\left[  2\right]  }$ we obtain%
\[
\mathcal{L}^{\left[  2\right]  }v^{\left[  2\right]  }=-\int_{\Gamma}%
p_{2}\left(  \alpha_{\delta}-\alpha_{e}\right)  \Delta_{\Gamma}u_{e,0|\Gamma
}\ v^{\left[  2\right]  }(m,0)\ d\Gamma.
\]
As a consequence,
\begin{align*}
&  \int_{\Gamma}\left[  \alpha_{i}\partial_{\mathbf{n}}u_{i,1|\Gamma}%
-p_{1}\left(  \alpha_{\delta}-\alpha_{1}\right)  \Delta_{\Gamma}u_{i,0|\Gamma
}\right]  v^{\left[  1\right]  }\left(  m,0\right)  \ d\Gamma\\
&  =\int_{\Gamma}\left[  \alpha_{e}\partial_{\mathbf{n}}u_{e,1|\Gamma}%
+p_{2}\left(  \alpha_{\delta}-\alpha_{e}\right)  \Delta_{\Gamma}u_{e,0|\Gamma
}\right]  v^{\left[  2\right]  }\left(  m,0\right)  \ d\Gamma.
\end{align*}
As $v^{\left[  1\right]  }(m,0)=v^{\left[  2\right]  }(m,0)$,
\begin{equation}
\alpha_{i}\partial_{\mathbf{n}}u_{i,1|\Gamma}-\alpha_{e}\partial_{\mathbf{n}%
}u_{e,1|\Gamma}=p_{1}(\alpha_{\delta}-\alpha_{i})\Delta_{\Gamma}u_{i,0|\Gamma
}+p_{2}(\alpha_{\delta}-\alpha_{e})\Delta_{\Gamma}u_{e,0|\Gamma}.\label{31}%
\end{equation}
It follows from (\ref{15}), (\ref{30a}), (\ref{31}) and Theorem \ref{theo1}
that $u_{1}$ is the unique solution of the following problem%
\[
\left\{
\begin{array}
[c]{ll}%
-div\left(  \alpha_{i}\nabla u_{i,1}\right)  =0 & \text{in }\Omega_{i},\\
-div\left(  \alpha_{e}\nabla u_{e,1}\right)  =0 & \text{in }\Omega_{e},\\
u_{e,1|\partial\Omega}=0 & \text{on }\partial\Omega,
\end{array}
\right.
\]
with transmission conditions on $\Gamma$%
\begin{align*}
u_{i,1|\Gamma}-u_{e,1|\Gamma} &  =p_{1}(1-\alpha_{i}\alpha_{\delta}%
^{-1})\partial_{\mathbf{n}}u_{i,0|\Gamma}+p_{2}(1-\alpha_{e}\alpha_{\delta
}^{-1})\partial_{\mathbf{n}}u_{e,0|\Gamma},\\
\alpha_{i}\partial_{\mathbf{n}}u_{i,1|\Gamma}-\alpha_{e}\partial_{\mathbf{n}%
}u_{e,1|\Gamma} &  =p_{1}(\alpha_{\delta}-\alpha_{i})\Delta_{\Gamma
}u_{i,0|\Gamma}+p_{2}(\alpha_{\delta}-\alpha_{e})\Delta_{\Gamma}u_{e,0|\Gamma
}.
\end{align*}
or
\begin{align*}
u_{i,1|\Gamma}-u_{e,1|\Gamma} &  =\left[  p_{1}(1-\alpha_{i}\alpha_{\delta
}^{-1})+p_{2}(\alpha_{i}\alpha_{e}^{-1}-\alpha_{i}\alpha_{\delta}%
^{-1})\right]  \partial_{\mathbf{n}}u_{i,0|\Gamma},\\
\alpha_{i}\partial_{\mathbf{n}}u_{i,1|\Gamma}-\alpha_{e}\partial_{\mathbf{n}%
}u_{e,1|\Gamma} &  =\left[  p_{1}(\alpha_{\delta}-\alpha_{i})+p_{2}%
(\alpha_{\delta}-\alpha_{e})\right]  \Delta_{\Gamma}u_{i,0|\Gamma}.
\end{align*}

\section{Convergence Theorem}

The process described in the previous section can be continued up to any order
provided that the data are sufficiently regular. We can also estimate the
error made by truncating the series after a finite number of terms. Let $n$ be
in $\mathbb{N},$ we set%
\[
u_{i,\delta}^{\left(  n\right)  }:=\sum\limits_{j=0}^{n}\delta^{j}%
u_{i,j},\text{ }\ u_{e,\delta}^{\left(  n\right)  }:=\sum\limits_{j=0}%
^{n}\delta^{j}u_{e,j}\text{ \ \ and\ }\ u_{d,\delta}^{\left(  n\right)
}:=\left\{
\begin{array}
[c]{c}%
u_{d_{1},\delta}^{\left(  n\right)  }:=\sum\limits_{j=0}^{n}\delta^{j}%
u_{d_{1},j}\text{ in }\Omega_{\delta,1,}\\
u_{d_{2},\delta}^{\left(  n\right)  }:=\sum\limits_{j=0}^{n}\delta^{j}%
u_{d_{2},j}\text{ in }\Omega_{\delta,2},
\end{array}
\right.
\]
where $u_{d_{\beta},j}(x):=\widetilde{u}_{d_{\beta},j}(m,\delta s_{\beta
}):=u_{j}^{\left[  \beta\right]  }(m,s_{\beta});$ $\forall x=\Phi_{\beta
}(m,s_{\beta})\in\Omega_{\delta,\beta}$.

\begin{theorem}
[Convergence Theorem]\label{theo2}For all integers $n$, there exists a
constant $c$ independent of $\delta$ such as
\[
\left\Vert u_{i,\delta}-u_{i,\delta}^{\left(  n\right)  }\right\Vert
_{H^{1}(\Omega_{i,\delta})}+\delta^{1/2}\left\Vert u_{d,\delta}-u_{d,\delta
}^{\left(  n\right)  }\right\Vert _{H^{1}(\Omega_{\delta})}+\left\Vert
u_{e,\delta}-u_{e,\delta}^{\left(  n\right)  }\right\Vert _{H^{1}%
(\Omega_{e,\delta})}\leq c\delta^{n+1}.
\]

\end{theorem}

\begin{proof}
Since $f$ is $\mathcal{C}^{\infty},$ all terms in Expansions (\ref{13}),
(\ref{14}) and (\ref{16}) up to order $n+1$ may be obtained from Equations
(\ref{23})-(\ref{26}). Let us define the remainders $R_{D_{1},n},R_{D_{2}%
,n},R_{1,n}$ and $R_{2,n}$ of Taylor expansions in the normal variable with
respect to $\delta$ up to order $n$ of $u_{i,\delta|\Gamma_{\delta,1}%
}^{\left(  n\right)  },u_{e,\delta|\Gamma_{\delta,2}}^{\left(  n\right)
},u_{i,\delta}^{\left(  n\right)  }$ and $u_{e,\delta}^{\left(  n\right)  }$
respectively by
\begin{align}
R_{D_{1},n}  & :=u_{i,\delta/\Gamma_{\delta,1}}^{\left(  n\right)  }%
-\sum\limits_{j=0}^{n}\sum\limits_{l=0}^{n-j}\frac{(-1)^{l}\delta^{j+l}}%
{l!}p_{1}^{l}\partial_{\mathbf{n}}^{l}u_{i,j|\Gamma},\label{33}\\
R_{D_{2},n}  & :=u_{e,\delta/\Gamma_{\delta,2}}^{\left(  n\right)  }%
-\sum\limits_{j=0}^{n}\sum\limits_{l=0}^{n-j}\frac{\delta^{j+l}}{l!}p_{2}%
^{l}\partial_{\mathbf{n}}^{l}u_{e,j|\Gamma},\label{34}\\
R_{1,n}^{\left[  1\right]  }  & :=\left(  u_{i,\delta}^{\left(  n\right)
}\right)  ^{\left[  1\right]  }-\sum\limits_{j=0}^{n}\sum\limits_{l=0}%
^{n-j}\frac{(s_{1})^{l}\delta^{j+l}}{l!}p_{1}^{l}\partial_{\mathbf{n}}%
^{l}u_{i,j|\Gamma}:=\left(  u_{i,\delta}^{\left(  n\right)  }\right)
^{\left[  1\right]  }-\sum\limits_{j=0}^{n}\delta^{j}U_{i,j},\label{35}\\
R_{2,n}^{\left[  2\right]  }  & :=\left(  u_{e,\delta}^{\left(  n\right)
}\right)  ^{\left[  2\right]  }-\sum\limits_{j=0}^{n}\sum\limits_{l=0}%
^{n-j}\frac{\left(  s_{2}\right)  ^{l}\delta^{j+l}}{l!}p_{2}^{l}%
\partial_{\mathbf{n}}^{l}u_{e,j|\Gamma}:=\left(  u_{e,\delta}^{\left(
n\right)  }\right)  ^{\left[  2\right]  }-\sum\limits_{j=0}^{n}\delta
^{j}U_{e,j},\label{36}%
\end{align}
where $s_{\beta}\in I_{\beta}.$ We shall rely on the following proposition to
show the estimates of the remainders $R_{D_{\beta},n}$ and $R_{\beta,n}$. The
steps of the proof are very similar to those given in \cite[Section
5]{SchTor2010}. We refer the reader to this paper.
\end{proof}

\begin{proposition}
\label{prop1}There exists a constant $c>0,$ independent of $\delta,$ such as%
\begin{align*}
\left\Vert \nabla R_{\beta,n}\right\Vert _{L^{2}(\Omega_{\delta,\beta})}  &
\leq c\delta^{n+1/2},\\
\left\Vert \nabla_{\Gamma}^{(j)}R_{D_{\beta},n}\right\Vert _{L^{2}(\Gamma)}
&  \leq c\delta^{n+1/2},\ \ \text{for }j=0,1.
\end{align*}
Moreover, there exists an extension $\mathcal{P}R$ of $R_{D_{\beta},n}$ into
$\Omega_{\delta}$ with%
\[
\partial_{\eta_{\beta}}\widetilde{\mathcal{P}R}\left(  m,\eta_{\beta}\right)
_{|\eta_{\beta}=(-1)^{\beta}\delta}=0\text{ and\ }\left\Vert \mathcal{P}%
R\right\Vert _{H^{1}(\Omega_{\delta})}\leq c\delta^{n}.
\]

\end{proposition}

\begin{proof}
[Continuation of the proof of Theorem \ref{theo2}]Let $r_{i,\delta}%
^{n},\ r_{d,\delta}^{n}$ and $\ r_{e,\delta}^{n}$ be the remainders made by
truncating Series (\ref{13}), (\ref{14}) and (\ref{16})%
\[
r_{i,\delta}^{n}:=u_{i,\delta}-u_{i,\delta}^{\left(  n\right)  }%
,\ r_{e,\delta}^{n}:=u_{e,\delta}-u_{e,\delta}^{\left(  n\right)
},\ r_{d,\delta}^{n}:=u_{d,\delta}-u_{d,\delta}^{\left(  n\right)  },
\]
and $\mathcal{L}_{\delta}$ be the linear form defined on $H_{0}^{1}(\Omega) $
\begin{align}
\mathcal{L}_{\delta}v &  :=\alpha_{i}\int_{\Omega_{i,\delta}}\nabla
r_{i,\delta}^{n}.\nabla v_{i}\ d\Omega_{i,\delta}+\alpha_{\delta}\int
_{\Omega_{\delta}}\nabla(r_{d,\delta}^{n}-\mathcal{P}R).\nabla v_{d}%
\ d\Omega_{\delta}\nonumber\\
&  +\alpha_{e}\int_{\Omega_{e,\delta}}\nabla r_{e,\delta}^{n}.\nabla
v_{e}\ d\Omega_{e,\delta},\label{36.1}%
\end{align}
in which $\mathcal{P}R$ is the extension function of $R_{D_{\beta},n}$ into
$\Omega_{\delta}.$ Using Green's formula in $\Omega_{i}$ and in $\Omega_{e}$
with the help of (\ref{15}), we obtain
\begin{align*}
\mathcal{L}_{\delta}v &  =-\alpha_{i}\int_{\Gamma}\left(  \partial
_{\mathbf{n}}u_{i,0|\Gamma}+\cdots+\delta^{n}\partial_{\mathbf{n}%
}u_{i,n|\Gamma}\right)  v_{i|\Gamma}\ d\Gamma\\
&  -\sum_{\beta=1}^{2}\left[  \alpha_{\delta}\delta a_{\delta}^{\left[
\beta\right]  }(u_{0}^{\left[  \beta\right]  }+\cdots+\delta^{n}u_{n}^{\left[
\beta\right]  },v^{\left[  \beta\right]  })\right] \\
&  +\alpha_{i}\int_{\Omega_{\delta,1}}\nabla u_{i,\delta}^{\left(  n\right)
}.\nabla v_{1}\ d\Omega_{\delta,1}+\alpha_{e}\int_{\Omega_{\delta,2}}\nabla
u_{e,\delta}^{\left(  n\right)  }.\nabla v_{2}\ d\Omega_{\delta,2}\\
&  +\alpha_{e}\int_{\Gamma}\left(  \partial_{\mathbf{n}}u_{e,0|\Gamma}%
+\cdots+\delta^{n}\partial_{\mathbf{n}}u_{e,n|\Gamma}\right)  v_{e|\Gamma
}\ d\Gamma-\alpha_{\delta}\int_{\Omega_{\delta}}\nabla\mathcal{P}R.\nabla
v_{d}\ d\Omega_{\delta}.
\end{align*}
It follows, from (\ref{33})-(\ref{36}), that%
\begin{align*}
\mathcal{L}_{\delta}v &  =\alpha_{i}\int_{\Omega_{i,\delta}}\nabla
R_{i,n}.\nabla v_{i}\ d\Omega_{i,\delta}+\alpha_{e}\int_{\Omega_{e,\delta}%
}\nabla R_{e,n}.\nabla v_{e}\ d\Omega_{e,\delta}\\
&  -\alpha_{\delta}\int_{\Omega_{\delta}}\nabla\mathcal{P}R.\nabla
v_{d}\ d\Omega_{\delta}-\sum_{\beta=1}^{2}\alpha_{\delta}\delta a_{\delta
}^{\left[  \beta\right]  }(u_{0}^{\left[  \beta\right]  }+\cdots+\delta
^{n}u_{n}^{\left[  \beta\right]  },v^{\left[  \beta\right]  })\\
&  +\alpha_{i}\delta a_{\delta}^{\left[  1\right]  }(U_{i,0}+\cdots+\delta
^{n}U_{i,n},v^{\left[  1\right]  })+\alpha_{e}\delta a_{\delta}^{\left[
2\right]  }(U_{e,0}+\cdots+\delta^{n}U_{e,n},v^{\left[  2\right]  })\\
&  -\alpha_{i}\int_{\Gamma}\left(  \partial_{\mathbf{n}_{\delta,1}%
}u_{i,0|\Gamma}+\cdots+\delta^{n}\partial_{\mathbf{n}}u_{i,n|\Gamma}\right)
v_{i|\Gamma}\ d\Gamma\\
&  +\alpha_{e}\int_{\Gamma}\left(  \partial_{\mathbf{n}}u_{e,0|\Gamma}%
+\cdots+\delta^{n}\partial_{\mathbf{n}}u_{e,n|\Gamma}\right)  v_{e|\Gamma
}\ d\Gamma.
\end{align*}
Now, we use the fact that $u_{0}^{\left[  \beta\right]  },\ldots
,u_{n+1}^{\left[  \beta\right]  },$ $(\beta=1,2)$ are solutions of Equations
(\ref{23})-(\ref{26}), we obtain%
\begin{align*}
\mathcal{L}_{\delta}v &  =\alpha_{\delta}\delta^{n+1}\sum_{\beta=1}%
^{2}\left\{  \delta^{-1}a_{0,2}^{\left[  \beta\right]  }\left(  u_{n+1}%
^{\left[  \beta\right]  },v^{\left[  \beta\right]  }\right)  -\left(
a_{2,2}^{\left[  \beta\right]  }+a_{0,1}^{\left[  \beta\right]  }\right)
\left(  u_{n}^{\left[  \beta\right]  },v^{\left[  \beta\right]  }\right)
\right. \\
&  -a_{1,1}^{\left[  \beta\right]  }\left(  u_{n-1}^{\left[  \beta\right]
}+\delta u_{n}^{\left[  \beta\right]  },v^{\left[  \beta\right]  }\right)
-a_{2,1}^{\left[  \beta\right]  }\left(  u_{n-2}^{\left[  \beta\right]
}+\delta u_{n-1}^{\left[  \beta\right]  }+\delta^{2}u_{n}^{\left[
\beta\right]  },v^{\left[  \beta\right]  }\right)  -\cdots\\
&  \left.  -a_{n-1,1}^{\left[  \beta\right]  }\left(  u_{1}^{\left[
\beta\right]  }+\cdots+\delta^{n-1}u_{n-1}^{\left[  \beta\right]  },v^{\left[
\beta\right]  }\right)  -r_{n}^{\left[  \beta\right]  }\left(  \delta
;u_{1}^{\left[  \beta\right]  }+\cdots+\delta^{n}u_{n}^{\left[  \beta\right]
},v^{\left[  \beta\right]  }\right)  \right\} \\
&  +\alpha_{i}\int_{\Omega_{i,\delta}}\nabla R_{i,n}.\nabla v_{i}%
\ d\Omega_{i,\delta}+\alpha_{e}\int_{\Omega_{e,\delta}}\nabla R_{e,n}.\nabla
v_{e}\ d\Omega_{e,\delta}-\alpha_{\delta}\int_{\Omega_{\delta}}\nabla
\mathcal{P}R.\nabla v_{d}\ d\Omega_{\delta}.
\end{align*}
By the estimates based on the explicit expressions of the bilinear form
$a_{k,l}^{\left[  \beta\right]  }(.,.)$ and those of Propositions \ref{prop1},
we have%
\begin{align*}
\left\vert \mathcal{L}_{\delta}v\right\vert  &  \leq c\delta^{n+1}\sum
_{\beta=1}^{2}\left(  \left\Vert \nabla_{\Gamma}v^{\left[  \beta\right]
}\right\Vert _{L^{2}(\Omega^{\beta})}+\delta^{-1}\left\Vert \partial
_{s_{\beta}}v^{\left[  \beta\right]  }\right\Vert _{L^{2}(\Omega^{\beta}%
)}+\left\Vert v^{\left[  \beta\right]  }\right\Vert _{L^{2}(\Omega^{\beta}%
)}\right) \\
&  +c\delta^{n-1/2}\left(  \left\Vert v_{i}\right\Vert _{H^{1}(\Omega
_{i,\delta})}+\sum_{\beta=1}^{2}\left\Vert v_{\beta}\right\Vert _{H^{1}%
(\Omega_{\delta,\beta})}+\left\Vert v_{e}\right\Vert _{H^{1}(\Omega_{e,\delta
})}\right)  .
\end{align*}
Since $\delta$ is small enough, we have, $\forall v\in H_{0}^{1}(\Omega),$%
\begin{align*}
\left\vert \mathcal{L}_{\delta}v\right\vert  &  \leq c\delta^{n+\frac{1}{2}%
}\sum_{\beta=1}^{2}\left(  \delta^{\frac{1}{2}}\left\Vert \nabla_{\Gamma
}v^{\left[  \beta\right]  }\right\Vert _{L^{2}(\Omega^{\beta})}+\delta
^{\frac{-1}{2}}\left\Vert \partial_{s_{\beta}}v^{\left[  \beta\right]
}\right\Vert _{L^{2}(\Omega^{\beta})}+\delta^{\frac{1}{2}}\left\Vert
v^{\left[  \beta\right]  }\right\Vert _{L^{2}(\Omega^{\beta})}\right) \\
&  +c\delta^{n-1/2}\left(  \left\Vert v_{i}\right\Vert _{H^{1}(\Omega
_{i,\delta})}+\sum_{\beta=1}^{2}\left\Vert v_{\beta}\right\Vert _{H^{1}%
(\Omega_{\delta,\beta})}+\left\Vert v_{e}\right\Vert _{H^{1}(\Omega_{e,\delta
})}\right)  .
\end{align*}
Therefore
\begin{equation}
\left\vert \mathcal{L}_{\delta}v\right\vert \leq c\delta^{n-\frac{1}{2}%
}\left\Vert v\right\Vert _{H^{1}(\Omega)}.\label{37}%
\end{equation}
Since $r^{n}:=\left(  r_{i,\delta}^{n},r_{d,\delta}^{n}-\mathcal{P}%
R,r_{e,\delta}^{n}\right)  $ is in $H_{0}^{1}(\Omega),$ we set in (\ref{36.1})
$v=r^{n},$ we obtain%
\[
\left\Vert r_{i,\delta}^{n}\right\Vert _{H^{1}(\Omega_{i,\delta})}+\left\Vert
r_{d,\delta}^{n}-\mathcal{P}R\right\Vert _{H^{1}(\Omega_{\delta})}+\left\Vert
r_{e,\delta}^{n}\right\Vert _{H^{1}(\Omega_{e,\delta})}\overset{(\ref{37}%
)}{\leq}c\delta^{n-\frac{1}{2}}.
\]
Thanks to Proposition \ref{prop1}, we find
\begin{equation}
\left\Vert r_{i,\delta}^{n}\right\Vert _{H^{1}(\Omega_{i,\delta})}+\left\Vert
r_{d,\delta}^{n}\right\Vert _{H^{1}(\Omega_{\delta})}+\left\Vert r_{e,\delta
}^{n}\right\Vert _{H^{1}(\Omega_{e,\delta})}\leq c\delta^{n-\frac{1}{2}%
}.\label{40}%
\end{equation}
Moreover, since $f_{i}$ and $f_{e}$ are $\mathcal{C}^{\infty}$ and for every
integer $j$, we have $\left\Vert u_{e,j}\right\Vert _{H^{1}(\Omega_{e,\delta
})}=O(1),$ $\left\Vert u_{i,j}\right\Vert _{H^{1}(\Omega_{i,\delta})}=O(1)$
and $\left\Vert u_{d_{\beta},j}\right\Vert _{H^{1}(\Omega_{\delta,\beta}%
)}=O(\delta^{-1/2})$, therefore (see \cite{tord-phd})%
\begin{gather*}
\left\Vert r_{i,\delta}^{n}\right\Vert _{H^{1}(\Omega_{i,\delta})}=\left\Vert
\delta^{n+1}u_{i,n+1}+\delta^{n+2}u_{i,n+2}+r_{i,\delta}^{n+2}\right\Vert
_{H^{1}(\Omega_{i,\delta})}\\
\overset{(\ref{40})}{\leq}c\delta^{n+1}+c\delta^{n+2}+c\delta^{n+3/2}\leq
c\delta^{n+1},\\
\left\Vert r_{e,\delta}^{n}\right\Vert _{H^{1}(\Omega_{e,\delta})}=\left\Vert
\delta^{n+1}u_{e,n+1}+\delta^{n+2}u_{e,n+2}+r_{e,\delta}^{n+2}\right\Vert
_{H^{1}(\Omega_{e,\delta})}\\
\overset{(\ref{40})}{\leq}c\delta^{n+1}+c\delta^{n+2}+c\delta^{n+3/2}\leq
c\delta^{n+1},\\
\left\Vert r_{d,\delta}^{n}\right\Vert _{H^{1}(\Omega_{\delta})}=\left\Vert
r_{d,\delta}^{n+1}+\delta^{n+1}u_{d,n+1}\right\Vert _{H^{1}(\Omega_{\delta}%
)}\overset{(\ref{40})}{\leq}c\delta^{n+1/2}+c\delta^{n+1/2}\\
\leq c\delta^{n+1/2}.
\end{gather*}
This completes the proof.
\end{proof}

\section{Approximate transmission conditions}

This section is devoted to the approximation of $u_{\delta}$ by a solution of
a problem modelling the effect of the thin layer with a precision of order two
in $\delta.$ We truncate the series defining the asymptotic expansions,
keeping only the first two terms%
\begin{align*}
u_{i,\delta} &  \simeq u_{i,\delta}^{\left(  1\right)  }=u_{i,0}+\delta
u_{i,1}\ \ \text{in }\Omega_{i},\\
u_{e,\delta} &  \simeq u_{e,\delta}^{\left(  1\right)  }=u_{e,0}+\delta
u_{e,1}\text{ \ in }\Omega_{e},\\
u_{d_{1},\delta}(x) &  \simeq u_{d_{1},\delta}^{\left(  1\right)  }%
(m,s_{1})=u_{0}^{\left[  1\right]  }(m,s_{1})+\delta u_{1}^{\left[  1\right]
}(m,s_{1}),\ \forall x=\Phi_{1}(m,s_{1})\in\Omega_{\delta,1},\\
u_{d_{2},\delta}(x) &  \simeq u_{d_{2},\delta}^{\left(  1\right)  }%
(m,s_{2}):=u_{0}^{\left[  2\right]  }(m,s_{2})+\delta u_{1}^{\left[  2\right]
}(m,s_{2}),\ \forall x=\Phi_{2}(m,s_{2})\in\Omega_{\delta,2},
\end{align*}
where $U_{\delta}^{\left(  1\right)  }:=\left(  u_{i,\delta}^{\left(
1\right)  },u_{e,\delta}^{\left(  1\right)  }\right)  $ is the solution of%
\begin{equation}
\left\{
\begin{array}
[c]{ll}%
-div\left(  \alpha_{i}\nabla u_{i,\delta}^{\left(  1\right)  }\right)
=f_{|\Omega_{i}} & \text{in }\Omega_{i},\\
-div\left(  \alpha_{e}\nabla u_{e,\delta}^{\left(  1\right)  }\right)
=f_{|\Omega_{e}} & \text{in }\Omega_{e},\\
u_{i,\delta|\Gamma}^{\left(  1\right)  }-u_{e,\delta|\Gamma}^{\left(
1\right)  }=\delta\mathcal{A}\left(  u_{i,\delta}^{\left(  1\right)  }\right)
-\delta^{2}\xi_{\delta} & \text{on }\Gamma,\\
\alpha_{i}\partial_{\mathbf{n}}u_{i,\delta|\Gamma}^{\left(  1\right)  }%
-\alpha_{e}\partial_{\mathbf{n}}u_{e,\delta|\Gamma}^{\left(  1\right)
}=\delta\mathcal{B}\left(  u_{i,\delta}^{\left(  1\right)  }\right)
-\delta^{2}\rho_{\delta} & \text{on}\ \Gamma,\\
U_{\delta|\partial\Omega}^{\left(  1\right)  }=0 & \text{on }\partial\Omega,
\end{array}
\right.  \label{41}%
\end{equation}
with%
\begin{align*}
\mathcal{A}\left(  u\right)   &  :=\left[  p_{1}(1-\alpha_{i}\alpha_{\delta
}^{-1})+p_{2}(\alpha_{i}\alpha_{e}^{-1}-\alpha_{i}\alpha_{\delta}%
^{-1})\right]  \left(  \partial_{\mathbf{n}}u_{|\Gamma}\right)  ,\\
\mathcal{B}\left(  u\right)   &  :=\left[  p_{1}(\alpha_{\delta}-\alpha
_{i})+p_{2}(\alpha_{\delta}-\alpha_{e})\right]  \Delta_{\Gamma}u_{|\Gamma},\\
\xi_{\delta} &  :=\left[  p_{1}(1-\alpha_{i}\alpha_{\delta}^{-1})+p_{2}%
(\alpha_{i}\alpha_{e}^{-1}-\alpha_{i}\alpha_{\delta}^{-1})\right]
\partial_{\mathbf{n}}u_{i,1|\Gamma},\\
\rho_{\delta} &  :=\left[  p_{1}(\alpha_{\delta}-\alpha_{i})+p_{2}%
(\alpha_{\delta}-\alpha_{e})\right]  \Delta_{\Gamma}u_{i,1|\Gamma}.
\end{align*}
Let $U_{\delta}^{ap}:=\left(  u_{i,\delta}^{ap},u_{e,\delta}^{ap}\right)  $ be
the solution of (\ref{41}) with $\rho_{\delta}=0$ and $\xi_{\delta}=0$. We
obtain a problem $\left(  \mathcal{P}_{\delta}^{ap}\right)  $ with
transmission conditions of order equal to that of the differential operator.
The new transmission conditions on $\Gamma$ are defined by
\begin{equation}
\left\{
\begin{array}
[c]{l}%
u_{i,\delta|\Gamma}^{ap}-u_{e,\delta|\Gamma}^{ap}=\delta\left[  p_{1}%
(1-\alpha_{i}\alpha_{\delta}^{-1})+p_{2}(\alpha_{i}\alpha_{e}^{-1}-\alpha
_{i}\alpha_{\delta}^{-1})\right]  \partial_{\mathbf{n}}u_{i,\delta|\Gamma
}^{ap},\\
\alpha_{i}\partial_{\mathbf{n}}u_{i,\delta|\Gamma}^{ap}-\alpha_{e}%
\partial_{\mathbf{n}}u_{e,\delta|\Gamma}^{ap}=\delta\left[  p_{1}%
(\alpha_{\delta}-\alpha_{i})+p_{2}(\alpha_{\delta}-\alpha_{e})\right]
\Delta_{\Gamma}u_{i,\delta|\Gamma}^{ap}.
\end{array}
\right.  \label{42}%
\end{equation}
However, the bilinear form associated to Problem $\left(  \mathcal{P}_{\delta
}^{ap}\right)  $ is neither positive nor negative. Then the existence and
uniqueness of the solution are not ensured by the Lax--Milgram lemma.
Therefore, we reformulate Problem $\left(  \mathcal{P}_{\delta}^{ap}\right)  $
into a nonlocal equation on the interface $\Gamma$ (cf. \cite{bonnvial}). A
direct use of transmission conditions (\ref{42}) leads to an operator which is
not self-adjoint. So, we choose the position of $\Gamma$ in such a way that
the jump of the trace of the solution on $\Gamma$ is null. We put
\[
p_{1}(1-\alpha_{i}\alpha_{\delta}^{-1})+p_{2}(\alpha_{i}\alpha_{e}^{-1}%
-\alpha_{i}\alpha_{\delta}^{-1})=0,
\]
we obtain%
\[
p_{1}=\frac{\alpha_{i}\left(  \alpha_{e}-\alpha_{\delta}\right)  }%
{\alpha_{\delta}\left(  \alpha_{e}-\alpha_{i}\right)  }\text{ and }p_{2}%
=\frac{\alpha_{e}\left(  \alpha_{\delta}-\alpha_{i}\right)  }{\alpha_{\delta
}\left(  \alpha_{e}-\alpha_{i}\right)  },
\]
which is valid only when $\alpha_{i}<\alpha_{\delta}<\alpha_{e}$ or
$\alpha_{e}<\alpha_{\delta}<\alpha_{i}.$ This corresponds to the case of
mid-diffusion. Transmission conditions (\ref{42}) become
\[
\left\{
\begin{array}
[c]{l}%
u_{i,\delta|\Gamma}^{ap}-u_{e,\delta|\Gamma}^{ap}=0,\\
\alpha_{i}\partial_{\mathbf{n}}u_{i,\delta|\Gamma}^{ap}-\alpha_{e}%
\partial_{\mathbf{n}}u_{e,\delta|\Gamma}^{ap}=\delta\dfrac{\left(  \alpha
_{e}-\alpha_{\delta}\right)  \left(  \alpha_{i}-\alpha_{\delta}\right)
}{\alpha_{\delta}}\Delta_{\Gamma}u_{i,\delta}^{ap}.
\end{array}
\right.
\]
After, we remove the right-hand side of Problem $\left(  \mathcal{P}_{\delta
}^{ap}\right)  $ by a standard lift: let $G$ be in $H_{0}^{1}(\Omega)$ such
that $-div\left(  \alpha\nabla G\right)  =f.$ Then the function $\Psi
=U_{\delta}^{ap}-G$ solves the following problem%
\[
\left\{
\begin{array}
[c]{ll}%
-div\left(  \alpha_{i}\nabla\Psi_{i}\right)  =0 & \text{in }\Omega_{i},\\
-div\left(  \alpha_{e}\nabla\Psi_{e}\right)  =0 & \text{in }\Omega_{e},\\
\Psi_{i|\Gamma}-\Psi_{e|\Gamma}=0 & \text{on }\Gamma,\\
\alpha_{i}\partial_{\mathbf{n}}\Psi_{i|\Gamma}-\alpha_{e}\partial_{\mathbf{n}%
}\Psi_{e|\Gamma}-\delta\dfrac{\left(  \alpha_{e}-\alpha_{\delta}\right)
\left(  \alpha_{i}-\alpha_{\delta}\right)  }{\alpha_{\delta}}\Delta_{\Gamma
}\Psi_{i|\Gamma}=g & \text{on}\ \Gamma,\\
\Psi_{e|\partial\Omega}=0 & \text{on }\partial\Omega,
\end{array}
\right.
\]
where $g=\left(  \alpha_{e}-\alpha_{i}\right)  \partial_{\mathbf{n}}%
G_{|\Gamma}+\delta\dfrac{\left(  \alpha_{e}-\alpha_{\delta}\right)  \left(
\alpha_{i}-\alpha_{\delta}\right)  }{\alpha_{\delta}}\Delta_{\Gamma}%
G_{|\Gamma}.$

We introduce the Steklov-Poicar\'{e} operators $S_{i}$ and $S_{e}$ (called
also Dirichlet-to-Neumann operators) defined from $H^{1/2}(\Gamma)$ onto
$H^{-1/2}(\Gamma)$ by $S_{i}\varphi:=\partial_{\mathbf{n}}u_{i|\Gamma},$ where
$u_{i}$ is the solution of the boundary value problem%
\[
\left\{
\begin{array}
[c]{ll}%
-\Delta u_{i}=0 & \text{in }\Omega_{i},\\
u_{i|\Gamma}=\varphi & \text{on }\Gamma,
\end{array}
\right.
\]
and by $S_{e}\psi:=\partial_{-\mathbf{n}}u_{e|\Gamma},$ where $u_{e}$ is the
solution of the boundary value problem
\[
\left\{
\begin{array}
[c]{ll}%
-\Delta u_{e}=0 & \text{in }\Omega_{e},\\
u_{e|\Gamma}=\psi & \text{on }\Gamma,\\
u_{e|\partial\Omega}=0 & \text{on }\partial\Omega.
\end{array}
\right.
\]
Then $\left(  \mathcal{P}_{\delta}^{ap}\right)  $ is equivalent to the
boundary equation
\[
\alpha_{i}S_{i}\omega+\alpha_{e}S_{e}\omega-\delta\dfrac{\left(  \alpha
_{e}-\alpha_{\delta}\right)  \left(  \alpha_{i}-\alpha_{\delta}\right)
}{\alpha_{\delta}}\Delta_{\Gamma}\omega=g\text{ on }\Gamma,
\]
where $\omega$ is the trace of $\Psi$ on the surface $\Gamma.$ We are now in
position to state the existence and uniqueness theorem, which proof is similar
to that of Theorem 2.5 in \cite{bonnvial}.

\begin{theorem}
\label{theo3}The operator $P_{\delta}:=\delta\dfrac{\left(  \alpha_{e}%
-\alpha_{\delta}\right)  \left(  \alpha_{i}-\alpha_{\delta}\right)  }%
{\alpha_{\delta}}\Delta_{\Gamma}-\alpha_{i}S_{i}-\alpha_{e}S_{e}$ is an
elliptic self-adjoint semi-bounded from below pseudodifferential operator of
order 2. Moreover, there exists series $\left(  \lambda_{n}\right)
_{n\in\mathbb{N}}$ growing to infinity such that for any $F\in H^{s}\left(
\Gamma\right)  $ with $s\in\mathbb{R}$, we have the following:

\begin{enumerate}
\item If $\ 0\notin\left(  \lambda_{n}\right)  _{n\in\mathbb{N}},$ then
equation $-P_{\delta}\omega=g$ admits a unique solution in $\mathcal{S}%
^{\prime}\left(  \Gamma\right)  $ which, in addition, belongs to
$H^{s+2}\left(  \Gamma\right)  ;$

\item If $\ 0\in\left(  \lambda_{n}\right)  _{n\in\mathbb{N}},$ then there is
either no solution or a complete affine finite dimensional space of
$H^{s+2}\left(  \Gamma\right)  $ solutions.
\end{enumerate}
\end{theorem}

Finally, we give an error estimate between the solution $u_{\delta}$ of
(\ref{1}) and the approximate solution $u_{\delta}^{ap}.$ Let us define
$u_{d_{\beta},\delta}^{ap}$ on $\Omega_{\delta,\beta}$
\begin{align*}
u_{d_{1},\delta}^{ap}\left(  x\right)   &  :=u_{d_{1},\delta}^{\left[
1\right]  ,ap}(m,s_{1}):=u_{i,\delta|\Gamma}^{ap}+\delta\frac{\alpha
_{i}\left(  \alpha_{e}-\alpha_{\delta}\right)  }{\alpha_{\delta}\left(
\alpha_{e}-\alpha_{i}\right)  }\left[  (s_{1}+1)\alpha_{i}\alpha_{\delta}%
^{-1}-1\right]  \partial_{\mathbf{n}}u_{i,\delta|\Gamma}^{ap},\\
u_{d_{2},\delta}^{ap}(x)  &  :=u_{d_{2},\delta}^{\left[  2\right]  ,ap}\left(
m,s_{2}\right)  :=u_{e,\delta|\Gamma}^{ap}+\delta\frac{\alpha_{e}\left(
\alpha_{\delta}-\alpha_{i}\right)  }{\alpha_{\delta}\left(  \alpha_{e}%
-\alpha_{i}\right)  }\left[  (s_{2}-1)\alpha_{e}\alpha_{\delta}^{-1}+1\right]
\partial_{\mathbf{n}}u_{e,\delta|\Gamma}^{ap},
\end{align*}
and let us denote by $u_{\delta}^{ap}$ the approximate solution defined on
$\Omega$
\[
u_{\delta}^{ap}:=\left\{
\begin{array}
[c]{cc}%
u_{i,\delta}^{ap} & \text{in }\Omega_{i,\delta},\\
u_{d_{\beta},\delta}^{ap} & \text{in }\Omega_{\delta,\beta},\\
u_{e,\delta}^{ap} & \text{in }\Omega_{e,\delta}.
\end{array}
\right.
\]
We can now formulate our main result.

\begin{theorem}
\label{theo4}There exists a constant $c$ independent of $\delta$ such as
\[
\left\Vert u_{i,\delta}-u_{i,\delta}^{ap}\right\Vert _{H^{1}\left(
\Omega_{i,\delta}\right)  }+\delta^{\frac{1}{2}}\sum\limits_{\beta=1}%
^{2}\left\Vert u_{d_{\beta},\delta}-u_{d_{\beta},\delta}^{ap}\right\Vert
_{H^{1}\left(  \Omega_{\delta,\beta}\right)  }+\left\Vert u_{e,\delta
}-u_{e,\delta}^{ap}\right\Vert _{H^{1}\left(  \Omega_{e,\delta}\right)  }\leq
c\delta^{2}.
\]

\end{theorem}

\begin{proof}
According to the Convergence Theorem, it is sufficient to estimate the error
$U_{\delta}^{ap}-U_{\delta}^{(1)}$. Therefore, as in \cite{vialphd}, we
perform an asymptotic expansion for $U_{\delta}^{ap}$. The ansatz%
\begin{equation}
U_{\delta}^{ap}=\sum\limits_{j\geq0}\delta^{j}w_{j},\label{46}%
\end{equation}
where $w_{j|\Omega_{e}}:=w_{e,j}$ and $w_{j|\Omega_{i}}:=w_{i,j}$, gives the
recurrence relations%
\[
\left\{
\begin{array}
[c]{ll}%
-div\left(  \alpha_{i}\nabla w_{i,j}\right)  =f_{|\Omega_{i}}\delta_{j,0} &
\text{in }\Omega_{i},\\
-div\left(  \alpha_{e}\nabla w_{e,j}\right)  =f_{|\Omega_{e}}\delta_{j,0} &
\text{in }\Omega_{e},\\
w_{i,j|\Gamma}-w_{e,j|\Gamma}=0 & \text{on }\Gamma,\\
\alpha_{i}\partial_{\mathbf{n}}w_{i,j|\Gamma}-\alpha_{e}\partial_{\mathbf{n}%
}w_{e,j|\Gamma}=\dfrac{\left(  \alpha_{e}-\alpha_{\delta}\right)  \left(
\alpha_{i}-\alpha_{\delta}\right)  }{\alpha_{\delta}}\Delta_{\Gamma
}w_{i,j-1|\Gamma} & \text{on}\ \Gamma,\\
w_{e,j|\partial\Omega}=0 & \text{on }\partial\Omega,
\end{array}
\right.
\]
with the convention that $w_{-1}=0.$ A simple calculation shows that the two
first terms $(w_{i,0},w_{e,0})$ and $(w_{i,1},w_{e,1})$ coincide with the two
first terms of (\ref{13}) and (\ref{14}). Furthermore, since $f_{e}$ and
$f_{i}$ are $\mathcal{C}^{\infty},$ each term of (\ref{46}) is bounded in
$H^{1}\left(  \Omega\right)  $. Then, by setting $\mathcal{R}_{w}:=U_{\delta
}^{ap}-w_{0}-\delta w_{1}-\delta^{2}w_{2},$ there exists $c>0$, such as
$\left\Vert \mathcal{R}_{w}\right\Vert _{H^{1}\left(  \Omega\right)  }\leq
c\delta^{2},$ which gives the desired result.
\end{proof}


\begin{thebibliography}{99}                                                                                               %
\bibitem {adn2}S. Agmon, A. Douglis and L. Nirenberg, Estimates near the
boundary for solutions of elliptic partial differential equations satisfying
general boundary conditions. II, Comm. Pure Appl. Math. 17 (1964) 35-92.

\bibitem {ammari-nedelec-1996}H. Ammari and J. C. N{\'e}d{\'e}lec, Sur les
conditions d'imp\'edance g\'en\'eralis\'ees pour les couches minces, C. R.
Acad. Sci. Paris S\'er. I Math. 322 (1996) 995-1000.

\bibitem {benlem08}A. Bendali and K. Lemrabet, Asymptotic analysis of the
scattering of a time-harmonic electromagnetic wave by a perfectly conducting
metal coated with a thin dielectric shell, Asymptot. Anal. 57 (2008) 199-227.

\bibitem {benlem96}A. Bendali and K. Lemrabet, The effect of a thin coating on
the scaterring of the time-harmonic wave for the Helmholtz equation, SIAM J.
Appl. Maths. 56 (1996) 1664-1693.

\bibitem {bonnvial}V. Bonnaillie-No\"el, M. Dambrine, F. H\'{e}rau and G.
Vial, On generalized Ventcel's type boundary conditions for Laplace operator
in a bounded domain, SIAM J. Math. Anal. 42, n%
${{}^\circ}$%
2 (2010) 931--945.

\bibitem {bouta1}K.E. Boutarene, Asymptotic analysis for a diffusion problem,
C. R. Math. Acad. Sci. Paris 349 (2011) 57-60.

\bibitem {brezis2011}H. Brezis, Functional analysis, {S}obolev spaces and
partial differential equations, Universitext, Springer, NewYork 2011.

\bibitem {delou}B. Delourme, H. Haddar and P. Joly, Approximate models for
wave propagation across thin periodic interfaces, J. Math. Pures Appl. 98
(2012) 28-71.

\bibitem {docarmo}M.P. Do Carmo, Differential geometry of curves and surfaces,
Prentice-Hall, Englewood Cliffs, NJ, 1976.

\bibitem {Durufle-Haddar-Joly-2006}M. Durufl\'{e}, H. Haddar and P. Joly, High
order generalized impedance boundary conditions in electromagnetic scattering
problems, C. R. Physique 7 (2006) 533-542.

\bibitem {engned}B. Engquist and J.C. N\'{e}d\'{e}lec, Effective boundary
conditions for acoustic and electromagnetic scattering in thin layers,
Research Report CMAP 278, Ecole Polytechnique, France 1993.

\bibitem {Li-Vogelius}Y.Y. Li and M. Vogelius, Gradient estimates for
solutions to divergence form elliptic equations with discontinuous
coefficients, Arch. Ration. Mech. Anal. 153 (2000) 91-151.

\bibitem {nedelec}J.C. N\'{e}d\'{e}lec, Acoustic and electromagnetic
equations, integral Representations for Harmonic Problems, Springer 2001.

\bibitem {clair-perrussel}R. Perrussel and C. Poignard, Asymptotic Expansion
of Steady-State Potential in a High Contrast Medium with a Thin Resistive
Layer, Applied Mathematics and Computation. 221 (2013) 48-65

\bibitem {clair2.2}C. Poignard, About the transmembrane voltage potential of a
biological cell in time-harmonic regime, in: Mathematical methods for imaging
and inverse problems, volume 26 of ESAIM Proc. EDP Sci. Les Ulis (2009) 162-179.

\bibitem {clair2.1}C. Poignard, Approximate transmission conditions through a
weakly oscillating thin layer, Math. Methods Appl. Sci. 32 (2009) 603-626.

\bibitem {clair-phd}C. Poignard, M\'{e}thodes asymptotiques pour le calcul des
champs \'{e}lectromagn\'{e}tiques dans des milieux \`{a} couches minces.
Application aux cellules biologiques, Th\`{e}se de Doctorat, Universit\'{e}
Claude Bernard-Lyon 1, 2006.

\bibitem {SchTor2010}K. Schmidt and S. Tordeux, Asymptotic modelling of
conductive thin sheets, Z. Angew. Math. Phys. 61 (2010) 603-626.

\bibitem {SchTor-magne2010}K. Schmidt and S. Tordeux, High order transmission
conditions for thin conductive sheets in magneto-quasistatics, ESAIM, Math.
Model. Numer. Anal. 45 No. 6 (2011) 1115-1140.

\bibitem {tord-phd}S. Tordeux, M\'{e}thodes asymptotiques pour la propagation
des ondes dans les milieux comportant des fentes, Ph.D. thesis, Universit\'{e}
de Versailles-Saint-Quentin, Yvelines, France 2004.

\bibitem {vialphd}G. Vial, Analyse multi-\'{e}chelle et conditions aux limites
approch\'{e}es pour un probl\`{e}me avec couche mince dans un domaine \`{a}
coin, Ph.D. thesis, Universit\'{e} de Renne 1, France 2003.
\end{thebibliography}
\end{document}